\documentclass{amsart}

\usepackage{a4}
\usepackage[utf8x]{inputenc}
\usepackage[american]{babel}
\usepackage{pifont}
\usepackage{xspace}

\usepackage{tikz}
\usetikzlibrary{calc} 
\usetikzlibrary{decorations} 
\usetikzlibrary{decorations.pathmorphing}
\usetikzlibrary{decorations.markings}
\usetikzlibrary{shadows}
\usetikzlibrary{shapes}
\usetikzlibrary{snakes}

\usepackage{hyperref}
\usepackage{blkarray} 
\usepackage{multirow}
\usepackage{amssymb}
\usepackage{amsmath}
\usepackage{enumitem}

\usepackage{epsfig}
\usepackage[linesnumbered,ruled,vlined,norelsize]{algorithm2e}

\hypersetup{
pdftitle={Linear rank-width of distance-hereditary graphs I. A polynomial-time algorithm}, 
pdfauthor={Isolde Adler, Mamadou Moustapha Kant\'e and O-joung Kwon}
}
\usepackage{mathabx}
\usepackage{subfig}
\theoremstyle{plain}

\newtheorem{theorem}{Theorem}[section]
\newtheorem{lemma}[theorem]{Lemma}
\newtheorem*{THMMAIN}{Theorem \ref{thm:mainalg}}

\newtheorem*{CORMAIN}{Corollary \ref{cor:matroidpathwidth}}
\newtheorem{corollary}[theorem]{Corollary}
\newtheorem{proposition}[theorem]{Proposition}

\newtheorem{fact}[theorem]{Fact}
\theoremstyle{remark}
\newtheorem*{remark}{Remark}
\theoremstyle{definition}

\newcommand\abs[1]{\lvert #1\rvert}

\newcommand\limb{\mathcal{L}}
\newcommand\limbhat{\mathcal{LG}}
\newcommand\limbtil{\mathcal{LC}}

\newcommand\rank{\operatorname{rank}}
\newcommand\rw{\operatorname{rw}}
\newcommand\lrw{\operatorname{lrw}}

\newcommand\cutrk{\operatorname{cutrk}}

\def\restrict#1#2{\mathchoice
              {\setbox1\hbox{${\displaystyle #1}_{\scriptstyle #2}$}
              \restrictionaux{#1}{#2}}
              {\setbox1\hbox{${\textstyle #1}_{\scriptstyle #2}$}
              \restrictionaux{#1}{#2}}
              {\setbox1\hbox{${\scriptstyle #1}_{\scriptscriptstyle #2}$}
              \restrictionaux{#1}{#2}}
              {\setbox1\hbox{${\scriptscriptstyle #1}_{\scriptscriptstyle #2}$}
              \restrictionaux{#1}{#2}}}
\def\restrictionaux#1#2{{#1\,\smash{\vrule height .8\ht1 depth .85\dp1}}_{\,#2}}

\newcommand\cO{\mathcal{O}}

\newcommand\bN{\mathbb{N}}

\newcommand\cI{\mathcal{I}}
\newcommand{\macro}[3]{\newcommand{#1}[#3]{#2}}
\macro{\lbwd}{\operatorname{pw}(#1)}{1}
\macro{\bwd}{\operatorname{bw}(#1)}{1}
\newcommand\push{\operatorname{push}}
\newcommand\pull{\operatorname{pull}}
\newcommand\maxnum{\eta}

\newcommand{\bag}[2]{\textsf{b}_{#1}(#2)}
\newcommand{\pbag}[2]{\textsf{pb}_{#1}(#2)}
\newcommand{\origin}[1]{\mathcal{G}[#1]}

\def\ie{\emph{i.e.}\xspace}

\def\cmp{\mathcal{T}}

\def\cT{\mathcal{T}}

\begin{document}

\title[LRW of  DH graphs I: A polynomial-time Algorithm]{Linear rank-width of distance-hereditary graphs I. A polynomial-time algorithm}

\author[I. Adler]{Isolde Adler}
\address{Institut f\"{u}r Informatik, Goethe-Universit\"{a}t, Frankfurt, Germany.}
\author[M.M. Kant\'e]{Mamadou Moustapha Kant\'e}
\address{Clermont-Universit\'e, Universit\'e Blaise Pascal, LIMOS, CNRS, France.}
\author[O. Kwon]{O-joung Kwon}
\address{Institute for Computer Science and Control, Hungarian Academy of Sciences, Hungary.}
  
  \email{iadler@cs.uni-frankfurt.de}
  \email{mamadou.kante@isima.fr}
  \email{ojoungkwon@gmail.com}
  
  \thanks{The first author is supported by the German Research Council, Project GalA, AD 411/1-1.}
  \thanks{The third author was partially supported by Basic Science Research
  Program through the National Research Foundation of Korea (NRF)
  funded by  the Ministry of Science, ICT \& Future Planning
  (2011-0011653) while the author was at Korea Advanced Institute of Science and Technology, and was also partially supported by ERC Starting Grant PARAMTIGHT (No. 280152).\\
 A preliminary version appeared in the proceedings of WG'14}

\begin{abstract} Linear rank-width is a linearized variation of rank-width, and it is deeply related to matroid path-width. 
 In this paper, we show that the linear rank-width of every $n$-vertex
  distance-hereditary graph, equivalently a graph of rank-width at most $1$, can be computed in time $\cO(n^2\cdot \log_2 n)$, and a linear layout witnessing the linear rank-width can be computed with the same time complexity.  
    As a corollary, we
  show that the path-width of every $n$-element matroid of branch-width at most $2$ can be computed in time $\cO(n^2\cdot \log_2 n)$,
  provided that the matroid is given by an independent set oracle.
    
    To establish this result, we present a characterization of the linear rank-width of distance-hereditary graphs
 in terms of their canonical split decompositions. 
   This characterization is similar to the known characterization of the path-width of forests
  given by Ellis, Sudborough, and Turner [The vertex separation and search number of a graph. {\em Inf. Comput.}, 113(1):50--79, 1994]. 
  However, different from forests, it is non-trivial to relate substructures of the canonical split decomposition of a graph
   with some substructures of the given graph.
   We introduce a notion of `limbs' of canonical split decompositions, which correspond to certain vertex-minors of the original graph, for the right characterization. 
%

 
\end{abstract}

\maketitle

\section{Introduction} 

{\em Rank-width}~\cite{OumS06} is a graph parameter introduced by Oum and Seymour with the goal of efficient approximation of the \emph{clique-width}~\cite{CourcelleO00} of a graph. {\em Linear
  rank-width} can be seen as the linearized variant of rank-width, and it is similar to path-width, 
  which can be seen as the linearized variant of tree-width.  While path-width is a well-studied notion, much
less is known about linear rank-width.  
Vertex-minor is a graph containment relation where rank-width and linear rank-width do not increase
when taking this operation.

Rank-width is related to matroid branch-width,
which has an important role in structural theory on matroids. 
We refer to the series of papers by Geelen, Gerards, and Whittle on the Matroid Minors Project~\cite{GeelenGW2002,GeelenGW2007} and Rota's Conjecture~\cite{GeelenGW2014} for more information on matroid branch-width.
It is known that the matroid branch-width (matroid path-width) of a binary matroid is equal to the rank-width (linear rank-width) of its fundamental graph plus one~\cite{Oum05}.
This equality can be further generalized to matroids over a fixed finite field with the finite field version of rank-width~\cite{Kante2012, KanteR2013}.
Hence new results on (linear) rank-width will immediately yield new
results on matroid branch-width or on matroid path-width.
In this paper, we will derive a complexity result for computing matroid path-width from linear rank-width.

Kashyap~\cite{Kashyap2008} showed that it is NP-hard to compute matroid path-width on binary matroids.
By reducing from matroid path-width, we can show that computing linear rank-width is NP-hard in general. 
Therefore it is natural to ask which graph classes allow for an efficient computation.  Until now, the
only known non-trivial result is for forests~\cite{AdlerK13}. 
Our main result is that distance-hereditary graphs allow a polynomial-time algorithm to compute linear rank-width.
 A graph $G$ is {\em distance-hereditary}, if for every pair of two vertices $u$ and $v$ of $G$, the distance between $u$ and $v$ in any connected
induced subgraph of $G$ containing both $u$ and $v$, is the same as the distance between $u$ and $v$ in $G$.   
Distance-hereditary graphs are exactly graphs of rank-width at most $1$~\cite{Oum05}, and include all forests and cographs. 
\begin{THMMAIN}
The linear rank-width of every $n$-vertex distance-hereditary graph can be computed in time $\cO(n^2\cdot \log_2 n)$. Moreover, 
a linear layout of the graph witnessing the linear rank-width can be computed with the same time complexity.
\end{THMMAIN}
In contrast, computing the path-width of distance-hereditary graphs is known to be NP-hard~\cite{THHD93}.
 Bodlaender and Kloks~\cite{BodlaenderK96} showed that 
 it is possible to compute the path-width of graphs of bounded tree-width in polynomial time.
The corresponding question for rank-width is still open:
Is there a polynomial-time algorithm to compute
the linear rank-width of graphs of rank-width at most $\ell$ (for fixed $\ell$)?
Since distance-hereditary graphs are exactly the graphs of rank-width at most
$1$, Theorem~\ref{thm:mainalg} is a first step towards an answer of this question.

%

 A direct consequence of Theorem \ref{thm:mainalg} is the possibility to compute the path-width of matroids with
branch-width at most $2$ in polynomial time. 

\begin{CORMAIN}
The path-width of every $n$-element matroid of branch-width at most $2$ can be computed in time $\mathcal{O}(n^2 \cdot \log_2 n)$,
provided that the matroid is given by an independent set oracle. Moreover, 
a linear layout of the matroid witnessing the path-width can be computed with the same time complexity.
\end{CORMAIN}

The main ingredient of our algorithm is a new characterization of the linear rank-width of distance-hereditary graphs (Theorem~\ref{thm:main}).  Our characterization makes use of the special structure of
canonical split decompositions~\cite{CunninghamE80} of distance-hereditary graphs.  Roughly, a canonical split decomposition decomposes a distance-hereditary graph in a tree-like fashion into complete graphs
and stars, and our characterization is recursive along the sub-decompositions of the split decomposition.

While a similar idea has been exploited in \cite{AdlerK13,EllisST94,MegiddoHGJP88} for other parameters, here we encounter a new problem. When we take a subgraph of a given split decomposition, the obtained split decomposition may have vertices that do not represent
vertices of the original graph. It is not at all obvious how to deal with these vertices in the recursive step.  We handle this by introducing {\em limbs} of canonical split decompositions, that
correspond to certain vertex-minors of the original graphs, and have the desired properties to allow our characterization.  We think that the notion of limbs may be useful in other contexts, too, and
hopefully, it can be extended to other graph classes and allow for further new efficient algorithms.

 The paper is structured as follows.  Section~\ref{sec:preliminaries} introduces the basic notions, in particular linear rank-width, vertex-minors, and split decompositions. In
 Section~\ref{sec:splitdecs}, we define limbs and its canonical decompositions, called canonical limbs, and show some basic properties. We use them in Section~\ref{sec:dh-lrw} for our characterization of the linear rank-width of distance-hereditary graphs.
 In Section~\ref{sec:lemmaonCL}, we establish essential properties of canonical limbs, which will be used to obtain the running time of our algorithm.
Section~\ref{sec:computing-lrw} presents the $\mathcal{O}(n^2\cdot \log_2 n)$-time algorithm for computing the linear rank-width of distance-hereditary graphs, and 
in Section~\ref{sec:matroid}, we obtain an algorithm for computing the path-width of matroids of branch-width at most $2$ as a corollary.
To obtain the running time, we need the fact that every $n$-vertex distance-hereditary graph $G$ has linear rank-width at most $\log_2 n$.
We prove it in Section~\ref{sec:upperbound}.

\section{Preliminaries}\label{sec:preliminaries}


In this paper, graphs are finite, simple and undirected, unless stated otherwise.  Our graph terminology is standard, see for instance \cite{Diestel05}. Let $G$ be a graph. We denote the vertex set of
$G$ by $V(G)$ and the edge set by $E(G)$. An edge between $x$ and $y$ is written $xy$ (equivalently $yx$).  For $X\subseteq V(G)$, we denote by $G[X]$ the subgraph of $G$ induced
by $X$, and let $G\setminus X:=G[V(G)\setminus X]$. 
For shortcut we write $G\setminus x$ for $G\setminus \{x\}$. For a vertex $x$ of $G$, let $N_G(x)$ be the set of \emph{neighbors} of $x$ in
$G$ and we call $|N_G(x)|$ the \emph{degree} of $x$ in $G$.
An edge $e$ of $G$ is called a \emph{cut-edge} if its removal increases the number of connected components of $G$.

A \emph{tree} is a connected acyclic graph. 
A \emph{leaf} of a tree is a vertex of degree one.  A \emph{path} is a tree where every vertex has degree at most two.  The \emph{length} of a path is the number of its edges.  A \emph{star}
is a tree with a distinguished vertex, called its \emph{center}, adjacent to all other vertices.  A \emph{complete graph} is a graph with all possible edges.  A graph $G$ is called
\emph{distance-hereditary} if for every pair of two vertices $x$ and $y$ of $G$ the distance of $x$ and $y$ in $G$ equals the distance of $x$ and $y$ in any connected induced subgraph
containing both $x$ and $y$~\cite{BandeltM86}.

\subsection{Linear rank-width and vertex-minors}\label{subsec:lrw-vm}

For sets $R$ and $C$, an \emph{$(R,C)$-matrix} is a matrix whose rows and columns are indexed by $R$ and $C$, respectively.
  For an $(R,C)$-matrix $M$,  $X\subseteq R$, and $Y\subseteq C$, 
  let $M[X,Y]$ be the submatrix of $M$ whose rows and columns are indexed by $X$ and $Y$, respectively.

\subsubsection*{\bf Linear rank-width} 

Let $G$ be a graph. We denote by $A_G$ the \emph{adjacency matrix} of $G$ over the binary field.  
The \emph{cut-rank function} of $G$ is a function $\cutrk_G:2^{V(G)}\rightarrow \mathbb{Z}$ where for each $X\subseteq V(G)$,  
\[\cutrk_G(X):=\rank(A_G[X,V(G)\setminus X]).\] 
A sequence $(x_1, \ldots, x_n)$ of the vertex set $V(G)$ is called a \emph{linear layout} of $G$. 
If $\abs{V(G)}\ge 2$, then the \emph{width} of a linear layout $(x_1,\ldots, x_n)$ of $G$ is defined as
\[\max_{1\le i\le n-1}\{\cutrk_G(\{x_1,\ldots,x_i\})\}.\]
The \emph{linear rank-width} of $G$, denoted by $\lrw(G)$, is defined as the minimum width over all linear layouts of $G$ if $\abs{V(G)}\ge 2$, and
otherwise, let $\lrw(G):=0$.

Caterpillars and complete graphs have linear rank-width at most $1$. Ganian~\cite{Ganian10} gave a characterization of the graphs of linear rank-width at most $1$, and call them \emph{thread graphs}. Adler and Kant\'{e}~\cite{AdlerK13} showed that linear rank-width and path-width coincide on forests, and therefore, there is a linear time algorithm to compute the linear rank-width of forests.  It is easy to see that the linear rank-width of a graph is the maximum over the linear
rank-widths of its connected components.

To obtain the bound presented in Theorem~\ref{thm:mainalg}, 
we will need the fact that the linear rank-width of an $n$-vertex distance-hereditary graph $G$ is at most $\log_2 n$.
In fact, we generally show that the linear rank-width of a graph with rank-width $k$ is at most $k\lfloor\log_2 n\rfloor$.
The proof scheme is similar to the one for path-width~\cite{BodlaenderGK91}.

A tree is \emph{subcubic} if it has at least two vertices and every inner vertex has degree~$3$. A \emph{rank-decomposition} of a graph $G$ is a pair $(T,L)$, where $T$ is a subcubic tree and $L$ is a bijection from the vertices of $G$ to the leaves of $T$. For an edge $e$ in $T$, $T\setminus e$ induces a partition $(X_{e} ,Y_{e} )$ of the leaves of $T$. The \emph{width} of an edge $e$ is defined as $\cutrk_{G} (L^{-1}(X_{e} ))$. The \emph{width} of a rank-decomposition $(T,L)$ is the maximum width over all edges of $T$. The \emph{rank-width} of $G$, denoted by $\rw(G)$, is the minimum width over all rank-decompositions of $G$. If $\abs{V(G)}\leq 1$, then $G$ admits no rank-decomposition and $\rw(G)=0$.

\begin{theorem}[Oum \cite{Oum05}]\label{thm:rankwidth1}
A graph is distance-hereditary if and only if it has rank-width at most $1$.
\end{theorem}

\begin{lemma}\label{lem:lrwtrivialbound}
Let $k$ be a positive integer and let $G$ be a graph of rank-width $k$. 
Then $\lrw(G)\le k \lfloor\log_{2}\abs{V(G)}\rfloor$.
\end{lemma}


Lemma~\ref{lem:lrwtrivialbound} will be proved in Section~\ref{sec:upperbound}.

\subsubsection*{\bf Vertex-minors}  For a graph $G$ and a vertex $x$ of $G$, the \emph{local complementation at $x$} of $G$ is an operation to replace the subgraph induced by the neighbors of $x$ with its
complement. The resulting graph is denoted by $G*x$.  If $H$ can be obtained from $G$ by applying a sequence of local complementations, then $G$ and $H$ are called \emph{locally equivalent}.  A graph $H$ is called a
\emph{vertex-minor} of a graph $G$ if $H$ can be obtained from $G$ by applying a sequence of local complementations and deletions of vertices. 

\begin{lemma}[Oum \cite{Oum05}] \label{lem:vm-rw} Let $G$ be a graph and let $x$ be a vertex of $G$. Then for every subset $X$ of $V(G)$, we have $\cutrk_G(X)=\cutrk_{G*x}(X)$. Therefore, every vertex-minor 
   $H$ of $G$ satisfies that $\lrw(H) \leq \lrw(G)$. 
\end{lemma}

For an edge $xy$ of $G$, let $W_1:=N_G(x)\cap N_G(y)$, $W_2:=(N_G(x)\setminus N_G(y))\setminus \{y\}$, and $W_3:=(N_G(y)\setminus N_G(x))\setminus \{x\}$.  The \emph{pivoting on $xy$} of $G$,
denoted by $G\wedge xy$, is the operation to complement the adjacencies between distinct sets $W_i$ and $W_j$, and swap the vertices $x$ and $y$.  It is known that $G\wedge
xy=G*x*y*x=G*y*x*y$ \cite{Oum05}.  See Figure~\ref{fig:pivoting} for an example.

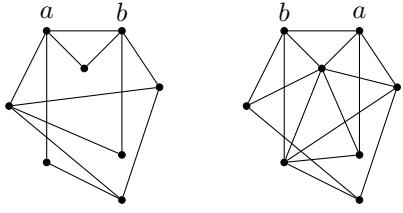
\begin{figure}[t]\centering
\tikzstyle{v}=[circle, draw, solid, fill=black, inner sep=0pt, minimum width=2.5pt]
\tikzset{photon/.style={decorate, decoration={snake}}}
\begin{tikzpicture}[scale=0.05]

\node [v] (a) at (10,50) {};
\node [v] (b) at (30,50) {};
\node [v] (c) at (0,30) {};
\node [v] (d) at (10,15) {};
\node [v] (e) at (30,17) {};
\node [v] (f) at (40,35) {};
\node [v] (g) at (30,5) {};
\node [v] (h) at (20,40) {};

\draw (c) -- (a) -- (b) -- (f);
\draw (d) -- (a);
\draw (e) -- (b);
\draw (f) -- (c);
\draw (c)-- (e);
\draw (g)--(c);
\draw (d) -- (g) -- (f);
\draw (a)--(h)--(b);

\draw (10,55) node{$a$};
\draw (30,55) node{$b$};

\end{tikzpicture}\qquad\quad
\begin{tikzpicture}[scale=0.05]

\node [v] (a) at (10,50) {};
\node [v] (b) at (30,50) {};
\node [v] (c) at (0,30) {};
\node [v] (d) at (10,15) {};
\node [v] (e) at (30,17) {};
\node [v] (f) at (40,35) {};
\node [v] (g) at (30,5) {};
\node [v] (h) at (20,40) {};

\draw (c) -- (a) -- (b) -- (f);
\draw (d) -- (a);
\draw (e) -- (b);
\draw (f) -- (d);
\draw (d)-- (e);

\draw (g)--(c);
\draw (d) -- (g) -- (f);
\draw (a)--(h)--(b);

\draw (c) -- (h);
\draw (d)-- (h);
\draw (e) -- (h);
\draw (f)-- (h);

\draw (10,55) node{$b$};
\draw (30,55) node{$a$};

\end{tikzpicture}
\caption{Pivoting an edge $ab$.}\label{fig:pivoting}
\end{figure}


We introduce some basic lemmas on local complementations, which will be used in several places.

\begin{lemma}\label{lem:changeloc} Let $G$ be a graph and $x, y\in V(G)$ such that $xy\notin E(G)$.  
Then $G*x*y=G*y*x$. \end{lemma}
\begin{proof}
It is straightforward as applying a local complementation at $x$ or $y$ does not change the neighbor sets of $x$ and $y$.
\end{proof}

\begin{lemma}\label{lem:changeloc2} Let $G$ be a graph and $x, y,z\in V(G)$ such that $xy,xz\notin E(G)$ and $yz\in E(G)$.  
Then $G*x\wedge yz=G\wedge yz *x$. \end{lemma}
\begin{proof}
By the definition of pivoting, $G*x\wedge yz=G*x*y*z*y$.
Note that $xy\notin E(G)$, $xz\notin E(G*y)$, and $xy\notin E(G*y*z)$.
Therefore, by Lemma~\ref{lem:changeloc}, $G*x*y*z*y=(G*y)*x*z*y=(G*y*z)*x*y=(G*y*z*y)*x=G\wedge yz*x$.
\end{proof}

\begin{lemma}[Oum \cite{Oum05}]\label{lem:equipiv} Let $G$ be a graph and $x,y,z\in V(G)$ such that $xy, yz\in E(G)$.  Then $G\wedge xy\wedge xz=G\wedge yz$. \end{lemma}

\subsection{Split decompositions and local complementations}\label{subsec:splitdecs} 

We will follow the definition of split decompositions in \cite{Bouchet88}.  
We notice that split decompositions are usually defined on connected graphs. For computing the linear rank-width of a distance-hereditary graph, it is enough to compute the linear rank-width of its connected components and take the maximum over all those values. Thus we will mostly assume that the given graph is connected in this paper, and use split decompositions in usual sense.

Let $G$ be a connected graph. A \emph{split} in $G$ is a vertex partition $(X,Y)$ of $G$ such that $|X|,|Y|\geq 2$ and
$\rank (A_G[X,Y]) = 1$. In other words, $(X,Y)$ is a split in $G$ if $|X|,|Y| \geq 2$ and there exist non-empty sets $X'\subseteq X$ and $Y'\subseteq Y$ such that $\{xy\in E(G) \mid x\in X, y\in Y\} =
\{xy \mid x\in X', y\in Y'\}$.  Notice that not all connected graphs have a split, and those that do not have a split are called \emph{prime} graphs.

A \emph{marked graph} $D$ is a connected graph $D$ with a set of edges $M(D)$, called \emph{marked edges}, that form a matching such that every edge in $M(D)$ is a cut-edge. The ends of the marked edges are called \emph{marked vertices}, and the components of $(V(D), E(D)\setminus M(D))$ are called \emph{bags} of $D$. The edges in
$E(D)\setminus M(D)$ are called \emph{unmarked edges}, and the vertices that are not marked vertices are called \emph{unmarked vertices}. 
If $(X,Y)$ is a split in $G$, then we construct a marked graph $D$ that consists of the vertex set $V(G) \cup \{x',y'\}$ for two distinct new vertices $x',y'\notin V(G)$ and the edge set $E(G[X]) \cup E(G[Y]) \cup \{x'y'\} \cup E'$ where we define $x'y'$
as marked and
\begin{align*}
E' &:= \{x'x\mid x\in X\ \textrm{and there exists $y\in Y$ such that $xy\in E(G)$}\} \cup\\ & \qquad \{y'y \mid y\in Y\ \textrm{and there exists $x\in X$ such that $xy\in E(G)$}\}.
\end{align*}
The marked graph $D$ is called a \emph{simple decomposition of} $G$. 

A \emph{split decomposition} of a connected graph $G$ is a marked graph $D$ defined inductively to be either $G$ or a marked graph
defined from a split decomposition $D'$ of $G$ by replacing a component $H$ of $(V(D'),E(D')\setminus M(D'))$ with a simple decomposition of $H$.  For a marked edge $xy$ in a split decomposition $D$, the \emph{recomposition
  of $D$ along $xy$} is the split decomposition $D':=(D\wedge xy) \setminus \{x,y\}$.  For a split decomposition $D$, let $\origin{D}$ denote the graph obtained from $D$ by recomposing all marked
edges. Note that if $D$ is a split decomposition of $G$, then $\origin{D}=G$.  Since each marked edge of a split decomposition $D$ is a cut-edge and all marked edges form a matching, if we contract all unmarked edges in $D$, then we
obtain a tree. We call it the \emph{decomposition tree of $G$ associated with $D$} and denote it by $T_D$.  
To distinguish the vertices of $T_D$ from the vertices of $G$ or $D$, 
the vertices of $T_D$ will be called \emph{nodes}. 
Obviously, the nodes of $T_D$ are in bijection with the bags of $D$.
Two bags of $D$ are
called \emph{neighbor bags} if their corresponding nodes in $T_D$ are adjacent.

A split decomposition $D$ of $G$ is called a \emph{canonical split decomposition} (or \emph{canonical decomposition} for short) if each bag of $D$ is either a prime, a star, or a complete graph, and $D$ is
not the refinement of a decomposition with the same property.  The following is due to Cunningham and Edmonds \cite{CunninghamE80}, and Dahlhaus \cite{Dahlhaus00}.

\begin{theorem}[Cunningham and Edmonds \cite{CunninghamE80}; Dahlhaus \cite{Dahlhaus00}] \label{thm:CED} Every connected graph $G$ has a unique canonical decomposition, up to isomorphism, and it can be computed in time $\cO(|V(G)|
  +|E(G)|)$.
\end{theorem}

From Theorem~\ref{thm:CED}, we can talk about only one canonical decomposition of a connected graph $G$ because all canonical decompositions of $G$ are isomorphic.

Let $D$ be a split decomposition of a connected graph $G$ with bags that are either primes, complete graphs or stars (it is not necessarily a canonical decomposition).  The \emph{type of a bag} of $D$ is either $P$, $K$, or $S$ depending on
whether it is a prime, a complete graph, or a star. The \emph{type of a marked edge} $uv$  is $AB$ where $A$ and $B$ are the types of the bags containing $u$ and $v$ respectively. If $A=S$ or $B=S$, then we can replace
$S$ by $S_p$ or $S_c$ depending on whether the end of the marked edge is a leaf or the center of the star. 

\begin{theorem}[Bouchet \cite{Bouchet88}]\label{thm:can-forbid} Let $D$ be a split decomposition of a connected graph with bags that are either primes, complete graphs, or stars. Then $D$ is a canonical decomposition if and only if it has no marked edge of type $KK$ or   $S_pS_c$.
\end{theorem}

We will use the following characterization of distance-hereditary graphs.

\begin{theorem}[Bocuhet \cite{Bouchet88}]\label{thm:Bouchet88} A connected graph is distance-hereditary if and only if each bag of its canonical decomposition is of type K or S. 
\end{theorem}

We now relate the split decompositions of a graph and the ones of its locally equivalent graphs. Let $D$ be a split decomposition of a connected graph.  A vertex $v$ of $D$ \emph{represents} an unmarked vertex $x$ (or is a
\emph{representative} of $x$) if either $v=x$ or there is a path of even length from $v$ to $x$ in $D$ starting with a marked edge such that marked edges and unmarked edges appear alternately in the
path.  Two unmarked vertices $x$ and $y$ are \emph{linked} in $D$ if there is a path from $x$ to $y$ in $D$ such that unmarked edges and marked edges appear alternately in the path.
    
\begin{lemma}\label{lem:represent} Let $D$ be a split decomposition of a connected graph.  Let $v'$ and $w'$ be two vertices in a same bag of $D$, and let $v$ and $w$ be two unmarked vertices of $D$ represented by
  $v'$ and $w'$, respectively.  The following are equivalent.
 \begin{enumerate}
 \item $v$ and $w$ are linked in $D$.
 \item $vw\in E(\origin{D})$.
 \item $v'w' \in E(D)$.
 \end{enumerate}
\end{lemma}
\begin{proof}
It is not hard to show that $v'$ and $w'$ are adjacent in $D$ 
if and only if there is an alternating path from $v$ to $w$ in $D$ from the definition of representativity.
Note that recomposing a marked edge in a split decomposition does not change the property that two unmarked vertices are linked, and the adjacency of two vertices in $\origin{D}$. It implies that $v$ and $w$ are linked in $D$ if and only if $vw\in E(\origin{D})$.
\end{proof}

A \emph{local complementation} at an unmarked vertex $x$ in a split decomposition $D$, denoted by $D*x$, is the operation to replace each bag $B$ containing a representative $w$ of $x$
with $B*w$. Observe that $D*x$ is a split decomposition of $\origin{D}*x$, and $M(D) = M(D*x)$.
Two split decompositions $D$ and $D'$ are \emph{locally equivalent}
if $D$ can be obtained from $D'$ by applying a sequence of local complementations at unmarked vertices.
	
\begin{lemma}[Bouchet \cite{Bouchet88}]\label{lem:localdecom} Let $D$ be the canonical decomposition of a connected graph. If $x$ is an unmarked vertex of $D$, then $D*x$ is the canonical decomposition of $\origin{D}*x$.
\end{lemma}

\begin{remark}\label{rem:localdecom} If $D$ is a canonical decomposition and $D'=D*x$ for some unmarked vertex $v$ of $D$, then $T_{D'}$ and $T_D$ are isomorphic because $M(D)=M(D')$. 
Thus, for every node $v$ of $T_D$ associated with a bag $B$ of $D$, 
its corresponding node $v'$ in $T_{D'}$ is associated in $D'$ with either 
\begin{enumerate}
\item $B$ if $x$ has no representative in $B$, or 
\item $B*w$ if $B$ has a representative $w$ of $v$. 
\end{enumerate}
  For easier arguments in several places, if $T_D$ is given for $D$, then we assume that $T_{D'}=T_D$ for every
  split decomposition $D'$ locally equivalent to $D$. 
  For a canonical decomposition $D$ and a node $v$ of its decomposition tree, we write $\bag{D}{v}$ to
  denote the bag of $D$ with which it is in correspondence.
\end{remark}

Let $x$ and $y$ be linked unmarked vertices in a split decomposition $D$, and let $P$ be the alternating path in $D$ linking $x$ and $y$.  Observe that each bag contains at most one unmarked edge in $P$.  Notice also
that if $B$ is a bag of type $S$ containing an unmarked edge of $P$, then the center of $B$ is a representative of either $x$ or $y$.  The \emph{pivoting on $xy$ of $D$}, denoted by $D\wedge xy$, is
the split decomposition obtained as follows: for each bag $B$ containing an unmarked edge of $P$, if $v, w\in V(B)$ represent respectively $x$ and $y$ in $D$, then we replace $B$ with $B\wedge vw$. (It is
worth noticing that by Lemma~\ref{lem:represent}, we have $vw\in E(B)$, hence $B\wedge vw$ is well-defined.)

\begin{lemma}\label{lem:pivotdecom} 
	Let $D$ be a split decomposition of a connected graph. 
	If $xy\in E(\origin{D})$, then $D\wedge xy=D*x*y*x$.  
\end{lemma}
\begin{proof} 
Since $xy\in E(\origin{D})$, by Lemma~\ref{lem:represent}, $x$ and $y$ are linked in $D$. It is easy to see that by the operation $D*x*y*x$, only
  the bags in the path from $x$ to $y$ are modified, and they are modified according to the definition of $D\wedge xy$. See Figure~\ref{fig:splitdecompositionpivoting} for an example of this procedure.
\end{proof}

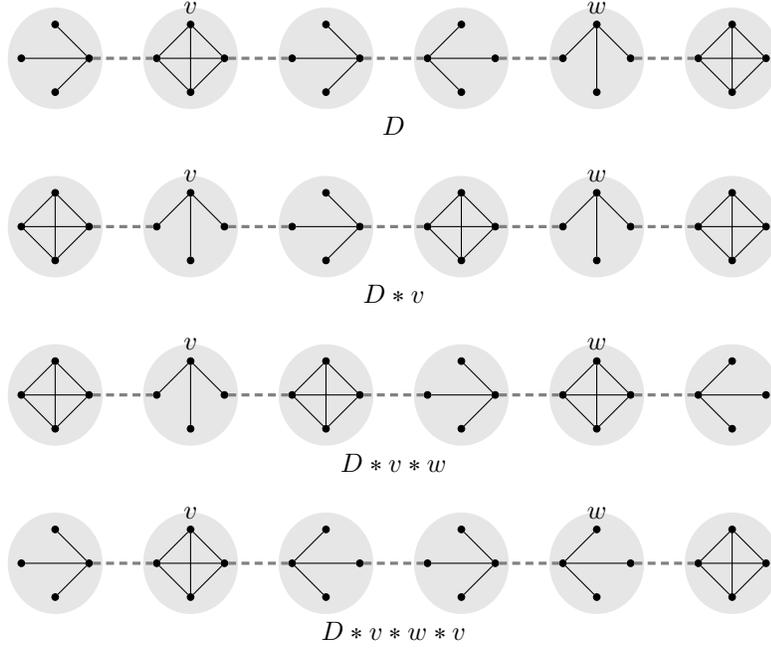
\begin{figure}[t]\centering
\tikzstyle{v}=[circle, draw, solid, fill=black, inner sep=0pt, minimum width=2.5pt]
 \tikzset{
    photon/.style={decorate,  draw=gray, very thick, , densely dashed},
}

\begin{tikzpicture}[scale=0.045]
\draw (30+60,20) node{$D$};

\path [fill=gray!20] (30-40,40) ellipse (14 and 15);
\node(a1) [v] at (30-40,50) {};
\node(a2) [v] at (30-40,30) {};
\node(a3) [v] at (20-40,40) {};
\node(a4) [v] at (40-40,40) {};
\draw(a4)--(a1);
\draw(a4)--(a2);
\draw(a4)--(a3);
\draw(a4)--(a4);
\path [fill=gray!20] (30,40) ellipse (14 and 15);

\node(b1) [v] at (30,50) {};
\node(b2) [v] at (30,30) {};
\node(b3) [v] at (20,40) {};
\node(b4) [v] at (40,40) {};
\draw(b1)--(b2);
\draw(b1)--(b3);
\draw(b1)--(b4);
\draw(b2)--(b3);
\draw(b2)--(b4);
\draw(b3)--(b4);
\draw (30 ,55) node{$v$};

\path [fill=gray!20] (30+40,40) ellipse (14 and 15);

\node(c1) [v] at (30+40,50) {};
\node(c2) [v] at (30+40,30) {};
\node(c3) [v] at (20+40,40) {};
\node(c4) [v] at (40+40,40) {};
\draw(c4)--(c1);
\draw(c4)--(c2);
\draw(c4)--(c3);
\draw(c4)--(c4);
\path [fill=gray!20] (30+80,40) ellipse (14 and 15);

\node(d1) [v] at (30+80,50) {};
\node(d2) [v] at (30+80,30) {};
\node(d3) [v] at (20+80,40) {};
\node(d4) [v] at (40+80,40) {};
\draw(d3)--(d1);
\draw(d3)--(d2);
\draw(d3)--(d3);
\draw(d3)--(d4);
\path [fill=gray!20] (30+120,40) ellipse (14 and 15);

\node(e1) [v] at (30+120,50) {};
\node(e2) [v] at (30+120,30) {};
\node(e3) [v] at (20+120,40) {};
\node(e4) [v] at (40+120,40) {};
\draw(e1)--(e1);
\draw(e1)--(e2);
\draw(e1)--(e3);
\draw(e1)--(e4);
\draw (30+120 ,55) node{$w$};

\path [fill=gray!20] (30+160,40) ellipse (14 and 15);

\node(f1) [v] at (30+160,50) {};
\node(f2) [v] at (30+160,30) {};
\node(f3) [v] at (20+160,40) {};
\node(f4) [v] at (40+160,40) {};
\draw(f1)--(f2);
\draw(f1)--(f3);
\draw(f1)--(f4);
\draw(f2)--(f3);
\draw(f2)--(f4);
\draw(f3)--(f4);


\draw[photon] (a4)--(b3);
\draw[photon] (b4)--(c3);
\draw[photon] (c4)--(d3);
\draw[photon] (d4)--(e3);
\draw[photon] (e4)--(f3);

\end{tikzpicture}
\vskip 0.2cm
\begin{tikzpicture}[scale=0.045]
\draw (30+60,20) node{$D*v$};

\path [fill=gray!20] (30-40,40) ellipse (14 and 15);
\node(a1) [v] at (30-40,50) {};
\node(a2) [v] at (30-40,30) {};
\node(a3) [v] at (20-40,40) {};
\node(a4) [v] at (40-40,40) {};
\draw(a1)--(a2);
\draw(a1)--(a3);
\draw(a1)--(a4);
\draw(a2)--(a3);
\draw(a2)--(a4);
\draw(a3)--(a4);

\path [fill=gray!20] (30,40) ellipse (14 and 15);

\node(b1) [v] at (30,50) {};
\node(b2) [v] at (30,30) {};
\node(b3) [v] at (20,40) {};
\node(b4) [v] at (40,40) {};
\draw(b1)--(b1);
\draw(b1)--(b2);
\draw(b1)--(b3);
\draw(b1)--(b4);

\draw (30 ,55) node{$v$};

\path [fill=gray!20] (30+40,40) ellipse (14 and 15);

\node(c1) [v] at (30+40,50) {};
\node(c2) [v] at (30+40,30) {};
\node(c3) [v] at (20+40,40) {};
\node(c4) [v] at (40+40,40) {};
\draw(c4)--(c1);
\draw(c4)--(c2);
\draw(c4)--(c3);
\draw(c4)--(c4);
\path [fill=gray!20] (30+80,40) ellipse (14 and 15);

\node(d1) [v] at (30+80,50) {};
\node(d2) [v] at (30+80,30) {};
\node(d3) [v] at (20+80,40) {};
\node(d4) [v] at (40+80,40) {};
\draw(d1)--(d2);
\draw(d1)--(d3);
\draw(d1)--(d4);
\draw(d2)--(d3);
\draw(d2)--(d4);
\draw(d3)--(d4);

\path [fill=gray!20] (30+120,40) ellipse (14 and 15);

\node(e1) [v] at (30+120,50) {};
\node(e2) [v] at (30+120,30) {};
\node(e3) [v] at (20+120,40) {};
\node(e4) [v] at (40+120,40) {};
\draw(e1)--(e1);
\draw(e1)--(e2);
\draw(e1)--(e3);
\draw(e1)--(e4);
\draw (30+120 ,55) node{$w$};

\path [fill=gray!20] (30+160,40) ellipse (14 and 15);

\node(f1) [v] at (30+160,50) {};
\node(f2) [v] at (30+160,30) {};
\node(f3) [v] at (20+160,40) {};
\node(f4) [v] at (40+160,40) {};
\draw(f1)--(f2);
\draw(f1)--(f3);
\draw(f1)--(f4);
\draw(f2)--(f3);
\draw(f2)--(f4);
\draw(f3)--(f4);


\draw[photon] (a4)--(b3);
\draw[photon] (b4)--(c3);
\draw[photon] (c4)--(d3);
\draw[photon] (d4)--(e3);
\draw[photon] (e4)--(f3);

\end{tikzpicture}
\vskip 0.2cm
\begin{tikzpicture}[scale=0.045]
\draw (30+60,20) node{$D*v*w$};

\path [fill=gray!20] (30-40,40) ellipse (14 and 15);
\node(a1) [v] at (30-40,50) {};
\node(a2) [v] at (30-40,30) {};
\node(a3) [v] at (20-40,40) {};
\node(a4) [v] at (40-40,40) {};
\draw(a1)--(a2);
\draw(a1)--(a3);
\draw(a1)--(a4);
\draw(a2)--(a3);
\draw(a2)--(a4);
\draw(a3)--(a4);
\path [fill=gray!20] (30,40) ellipse (14 and 15);

\node(b1) [v] at (30,50) {};
\node(b2) [v] at (30,30) {};
\node(b3) [v] at (20,40) {};
\node(b4) [v] at (40,40) {};
\draw(b1)--(b1);
\draw(b1)--(b2);
\draw(b1)--(b3);
\draw(b1)--(b4);
\draw (30 ,55) node{$v$};

\path [fill=gray!20] (30+40,40) ellipse (14 and 15);

\node(c1) [v] at (30+40,50) {};
\node(c2) [v] at (30+40,30) {};
\node(c3) [v] at (20+40,40) {};
\node(c4) [v] at (40+40,40) {};
\draw(c1)--(c2);
\draw(c1)--(c3);
\draw(c1)--(c4);
\draw(c2)--(c3);
\draw(c2)--(c4);
\draw(c3)--(c4);
\path [fill=gray!20] (30+80,40) ellipse (14 and 15);

\node(d1) [v] at (30+80,50) {};
\node(d2) [v] at (30+80,30) {};
\node(d3) [v] at (20+80,40) {};
\node(d4) [v] at (40+80,40) {};
\draw(d4)--(d1);
\draw(d4)--(d2);
\draw(d4)--(d3);
\draw(d4)--(d4);
\path [fill=gray!20] (30+120,40) ellipse (14 and 15);

\node(e1) [v] at (30+120,50) {};
\node(e2) [v] at (30+120,30) {};
\node(e3) [v] at (20+120,40) {};
\node(e4) [v] at (40+120,40) {};
\draw(e1)--(e2);
\draw(e1)--(e3);
\draw(e1)--(e4);
\draw(e2)--(e3);
\draw(e2)--(e4);
\draw(e3)--(e4);
\draw (30+120 ,55) node{$w$};

\path [fill=gray!20] (30+160,40) ellipse (14 and 15);

\node(f1) [v] at (30+160,50) {};
\node(f2) [v] at (30+160,30) {};
\node(f3) [v] at (20+160,40) {};
\node(f4) [v] at (40+160,40) {};
\draw(f3)--(f1);
\draw(f3)--(f2);
\draw(f3)--(f3);
\draw(f3)--(f4);


\draw[photon] (a4)--(b3);
\draw[photon] (b4)--(c3);
\draw[photon] (c4)--(d3);
\draw[photon] (d4)--(e3);
\draw[photon] (e4)--(f3);

\end{tikzpicture}
\vskip 0.2cm
\begin{tikzpicture}[scale=0.045]
\draw (30+60,20) node{$D*v*w*v$};

\path [fill=gray!20] (30-40,40) ellipse (14 and 15);
\node(a1) [v] at (30-40,50) {};
\node(a2) [v] at (30-40,30) {};
\node(a3) [v] at (20-40,40) {};
\node(a4) [v] at (40-40,40) {};
\draw(a4)--(a1);
\draw(a4)--(a2);
\draw(a4)--(a3);
\draw(a4)--(a4);
\path [fill=gray!20] (30,40) ellipse (14 and 15);

\node(b1) [v] at (30,50) {};
\node(b2) [v] at (30,30) {};
\node(b3) [v] at (20,40) {};
\node(b4) [v] at (40,40) {};
\draw(b1)--(b2);
\draw(b1)--(b3);
\draw(b1)--(b4);
\draw(b2)--(b3);
\draw(b2)--(b4);
\draw(b3)--(b4);
\draw (30 ,55) node{$v$};

\path [fill=gray!20] (30+40,40) ellipse (14 and 15);

\node(c1) [v] at (30+40,50) {};
\node(c2) [v] at (30+40,30) {};
\node(c3) [v] at (20+40,40) {};
\node(c4) [v] at (40+40,40) {};
\draw(c3)--(c1);
\draw(c3)--(c2);
\draw(c3)--(c3);
\draw(c3)--(c4);
\path [fill=gray!20] (30+80,40) ellipse (14 and 15);

\node(d1) [v] at (30+80,50) {};
\node(d2) [v] at (30+80,30) {};
\node(d3) [v] at (20+80,40) {};
\node(d4) [v] at (40+80,40) {};
\draw(d4)--(d1);
\draw(d4)--(d2);
\draw(d4)--(d3);
\draw(d4)--(d4);
\path [fill=gray!20] (30+120,40) ellipse (14 and 15);

\node(e1) [v] at (30+120,50) {};
\node(e2) [v] at (30+120,30) {};
\node(e3) [v] at (20+120,40) {};
\node(e4) [v] at (40+120,40) {};
\draw(e3)--(e1);
\draw(e3)--(e2);
\draw(e3)--(e3);
\draw(e3)--(e4);
\draw (30+120 ,55) node{$w$};

\path [fill=gray!20] (30+160,40) ellipse (14 and 15);

\node(f1) [v] at (30+160,50) {};
\node(f2) [v] at (30+160,30) {};
\node(f3) [v] at (20+160,40) {};
\node(f4) [v] at (40+160,40) {};
\draw(f1)--(f2);
\draw(f1)--(f3);
\draw(f1)--(f4);
\draw(f2)--(f3);
\draw(f2)--(f4);
\draw(f3)--(f4);


\draw[photon] (a4)--(b3);
\draw[photon] (b4)--(c3);
\draw[photon] (c4)--(d3);
\draw[photon] (d4)--(e3);
\draw[photon] (e4)--(f3);

\end{tikzpicture}

\caption{The split decomposition $D*v*w*v$, which is the same as $D\wedge vw$.}\label{fig:splitdecompositionpivoting}
\end{figure} 


 As a corollary of Lemmas \ref{lem:localdecom} and \ref{lem:pivotdecom}, we get the following.
	
\begin{corollary}\label{cor:pivotdecom} 
	Let $D$ be the canonical decomposition of a connected graph. If $xy\in E(\origin{D})$, then $D\wedge xy$ is the canonical decomposition of $\origin{D}\wedge xy$.  
\end{corollary}

The following are split decomposition versions of Lemma~\ref{lem:changeloc}, \ref{lem:changeloc2}, \ref{lem:equipiv}, and they can be easily verified in a same way.

\begin{lemma}\label{lem:decompverlemma}
Let $D$ be the canonical decomposition of a connected graph. The following are satisfied.
\begin{enumerate}
\item If $x, y$ are unmarked vertices of $D$ that are not linked, then
$D*x*y=D*y*x$.
\item If $x, y,z$ are unmarked vertices of $D$ such that $x$ is linked to neither $y$ nor $z$, and $y$ and $z$ are linked,  
then $D*x\wedge yz=D\wedge yz *x$.
\item If $x,y,z$ are unmarked vertices of $D$ such that $y$ is linked to both $x$ and $z$, then $D\wedge xy\wedge xz=D\wedge yz$.
\end{enumerate}
\end{lemma}

For a bag $B$ of $D$ and a component $T$ of $D\setminus V(B)$, let us denote by $\zeta_b(D,B,T)$ and $\zeta_t(D,B,T)$ the adjacent marked vertices of $D$ that are in $V(B)$ and in $V(T)$ respectively.
Observe that $\zeta_t(D,B,T)$ is not incident with any marked edge in $T$. So, when we take a sub-decomposition $T$ from $D$, we regard $\zeta_t(D,B,T)$ as an unmarked vertex of $T$.  


\section{Limbs in canonical decompositions}\label{sec:splitdecs}

We define the notion of \emph{limb} that is the key ingredient in our characterization.  Intuitively, a limb of a canonical decomposition 
is a modification of its sub-decomposition satisfying the property that
if two canonical decompositions $D$ and $D'$ are locally equivalent, then 
the limbs obtained from $D$ and $D'$ on the same vertex set are again locally equivalent. This property allows us to characterize linear rank-width in a right way. To avoid overloading the statements
in this section let us fix $D$ the canonical decomposition of a connected distance-hereditary graph $G$. 
We recall from Theorems \ref{thm:can-forbid} and \ref{thm:Bouchet88} that each bag of $D$ is of type K or S, and marked edges of types KK or
$S_pS_c$ do not occur.

For an unmarked vertex $y$ in $D$ and a bag $B$ of $D$ containing a marked vertex that represents $y$,
let $T$ be the component of $D\setminus V(B)$ containing $y$, and 
let $v$ and $w$ be adjacent marked vertices of $D$ where $v\in V(T)$ and $w\in V(B)$.
We define the \emph{limb } $\limb:=\limb_D[B,y]$ with respect to $B$ and $y$ as follows:
\begin{enumerate}
\item if $B$ is of type $K$, then $\limb:=T*v\setminus v$,       
\item if $B$ is of type $S$ and $w$ is a leaf, then $\limb:=T\setminus v$,
\item if $B$ is of type $S$ and $w$ is the center, then $\limb:=T\wedge vy \setminus v$.
\end{enumerate}
Since $v$ becomes an unmarked vertex in $T$, the limb is well-defined and it is a split decomposition.  While $T$ is a canonical decomposition, $\limb$ may not be a canonical decomposition at all, because deleting $v$ may create a
bag of size $2$.  We analyze the cases when such a bag appears, and describe how to transform it into a canonical decomposition.

Suppose that a bag $B'$ of size $2$ appears in $\limb$ by deleting $v$.
If $B'$ has no adjacent bags in $\limb$, then $B'$ itself is a canonical decomposition.
Otherwise we have two cases.
\begin{enumerate}
\item ($B'$ has one neighbor bag $B_1$.) \\
If $v_1\in V(B_1)$ is the the marked vertex adjacent to a vertex of $B'$
and $r$ is the unmarked vertex of $B'$ in $\mathcal{L}$,
then we can transform the limb into a canonical decomposition by removing the bag $B'$ and replacing $v_1$ with $r$.

\item ($B'$ has two neighbor bags $B_1$ and $B_2$.)\\
  If $v_1\in V(B_1)$ and $v_2\in V(B_2)$ are the two marked vertices that are adjacent to the two marked vertices of $B'$, then we can first transform the limb into another decomposition by removing $B'$ and adding a
  marked edge $v_1v_2$. If the new marked edge $v_1v_2$ is of type KK or $S_pS_c$, then by recomposing along $v_1v_2$, we finally transform the limb into a canonical decomposition.
\end{enumerate}

  Let $\limbtil_D[B,y]$ be the canonical decomposition obtained from $\limb_D[B,y]$ and we call it the \emph{canonical limb}. 
  Let
$\limbhat_D[B,y]$ be the graph obtained from $\limb_D[B,y]$ by recomposing all marked edges.
See Figure~\ref{fig2} for an example of canonical limbs.

\begin{figure}[t]
  \tikzstyle{v}=[circle, draw, solid, fill=black, inner sep=0pt, minimum width=3pt]
  \tikzset{
    photon/.style={decorate,  draw=gray, very thick, , densely dashed},
}
\quad
 \begin{tikzpicture}[scale=0.32]
\draw (5.2,4.7) circle (1.2);
    
    \node [v] (a) at (3.5,10.7) {};
    \node [v] (b) at (5.5,10.7) {};
    \node [v] (a2) at (7.5,10.7) {};
    \node [v] (b2) at (9.5,10.7) {};
    \node [v] (c2) at (8.5,11.2) {};
    \node [v] (c) at (5+3-.7,8) {};
    \node [v] (h) at (8.5,3) {};
    \node [v] (i) at (8.8,5) {};
	
    \node [v] (d) at (0.5,1.5) {};
    \node [v] (e) at (2.5,1.5) {};
    \node [v] (f) at (1.5,0) {};
    \node [v] (g) at (2,6) {};
    \node [v] (g1) at (0.8,5) {};

	\node [v] (y2) at (1.5,2.7) {};
	\node [v] (y1) at (2,4) {};
	\draw[photon] (y1) -- (y2);

	\foreach \x in {d,e,f}
	\draw (y2) -- (\x);
	
    \node [v] (x) at (3,4.5) {};
    \node [v] (y) at (4.5,4.5) {};
	\draw (y1)--(x);
	\node [v] (x1) at (6, 4.5) {};
	\node [v] (x2) at (7.5, 4.5) {};
		
	\foreach \x in {y}
	\draw (x1) -- (\x);
	\foreach \x in {i}
	\draw (x2) -- (\x);
	\draw[photon] (x1) --(x2);
	\draw(h)--(i);
    
    \node [v] (x3) at (5-.8, 9.7) {};
    \node [v] (x32) at (8.2, 9.7) {};
    \draw[photon] (c)--(x32);
    \foreach \x in {a2,c2,x32}
    \draw (b2) -- (\x);

    \node [v] (x4) at (6-.8, 8) {};
    \foreach \x in {a,b}
    \draw (x3) -- (\x);
	\draw[photon] (x3) -- (x4);
	\node [v] (x5) at (6, 7) {};
    \node [v] (x6) at (5.5, 5.5) {};
    \foreach \x in {x4,c}
    \draw (x5) -- (\x);
    \foreach \x in {x1,x1}
    \draw (x6) -- (\x);
	\draw[photon] (x5)--(x6);
	\draw (y)--(x6);
							
	\foreach \x in {g, g1}
	\draw (x)--(\x);
	\draw[photon] (x)--(y);
                        
    \foreach \x in {b}
    \draw (a)--(\x);
    \foreach \x in {f,e}
    \draw (d)--(\x);
    \foreach \x in {f}
    \draw (e)--(\x);

 \node [label=$B$] at (5.3,1.8) {};
\node [label=$(a)$] at (5.3,-3.5) {};

\draw(11, -2)-- (11, 10);
   \end{tikzpicture}
   \quad
 \begin{tikzpicture}[scale=0.32]
    
    \node [v] (a) at (3.5,10.7) {};
    \node [v] (b) at (5.5,10.7) {};
    \node [v] (a2) at (7.5,10.7) {};
    \node [v] (b2) at (9.5,10.7) {};
    \node [v] (c2) at (8.5,11.2) {};
    \node [v] (c) at (5+3-.7,8) {};
    \node [v] (h) at (8.5,3) {};
    \node [v] (i) at (8.8,5) {};
	
    \node [v] (d) at (0.5,1.5) {};
    \node [v] (e) at (2.5,1.5) {};
    \node [v] (f) at (1.5,0) {};
    \node [v] (g) at (2,6) {};
   \node [v] (g1) at (0.8,5) {};

	\node [v] (y2) at (1.5,2.7) {};
	\node [v] (y1) at (2,4) {};
	\draw[photon] (y1) -- (y2);
	\foreach \x in {d,e,f}
	\draw (y2) -- (\x);

    \node [v] (x3) at (5-.8, 9.7) {};
    \node [v] (x32) at (8.2, 9.7) {};
    \draw[photon] (c)--(x32);
    \foreach \x in {a2,c2,x32}
    \draw (b2) -- (\x);

          \node [v] (x4) at (6-.8, 8) {};
    \foreach \x in {a,b}
    \draw (x3) -- (\x);
	\draw[photon] (x3) -- (x4);
\draw (y1)--(g1)--(g);
\draw (y1)--(g); \draw (h)--(i); \draw (x4)--(c);
\node [label=$(b)$] at (5.3,-3.5) {};

\draw(11, -2)-- (11, 10);

  \end{tikzpicture}
  \quad
    \begin{tikzpicture}[scale=0.32]
    
    \node [v] (a) at (3.5,10.7) {};
    \node [v] (b2) at (4.5,11.2) {};
    \node [v] (b3) at (5.5,11.2) {};
    \node [v] (b) at (6.5,10.7) {};
    \node [v] (h) at (8.5,3) {};
    \node [v] (i) at (8.8,5) {};
	
    \node [v] (d) at (0.5,1.5) {};
    \node [v] (e) at (2.5,1.5) {};
    \node [v] (f) at (1.5,0) {};
    \node [v] (g) at (2,6) {};
 \node [v] (g1) at (0.8,5) {};

	\node [v] (y2) at (1.5,2.7) {};
	\node [v] (y1) at (2,4) {};
	\draw[photon] (y1) -- (y2);
	\foreach \x in {d,e,f}
	\draw (y2) -- (\x);

    \node [v] (x3) at (5, 9.7) {};
    \foreach \x in {a,b,b2,b3}
    \draw (x3) -- (\x);
\draw (y1)--(g1)--(g);
\draw (y1)--(g); \draw (h)--(i);
\node [label=$(c)$] at (5.3,-3.5) {};

  \end{tikzpicture}
  
  \caption{In $(a)$, we have a canonical decomposition $D$ of a distance-hereditary graph with a bag $B$. The dashed edges are marked edges of $D$. In $(b)$, we have limbs associated with the components of $D\setminus V(B)$. The canonical limbs associated with limbs are shown in $(c)$. }
  \label{fig2}

\end{figure}
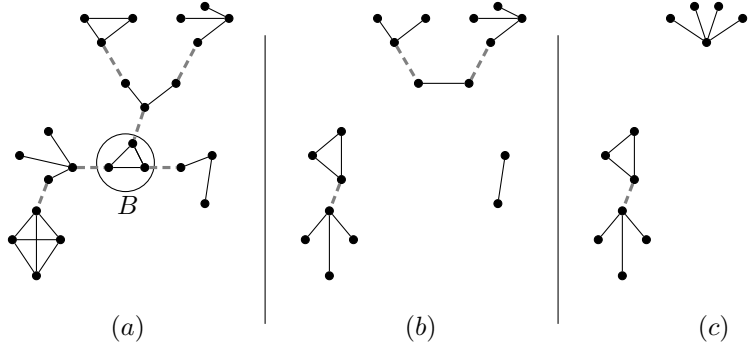

\begin{lemma}\label{lem:connected} Let $B$ be a bag of $D$. If an unmarked vertex $y$ of $D$ is represented by a marked vertex of $B$, then $\limb_{D}[B,y]$ is connected.
\end{lemma}

\begin{proof} 
Let $T$ be the component of $D\setminus V(B)$ containing $y$, and
$v:=\zeta_t(D,B,T)$, and
$B'$ be the bag of $D$ containing $v$.  Since local complementations maintain
  connectedness, it suffices to verify that $V(B')\setminus \{v\}$ induces a connected subgraph in $\limb_{D}[B,y]$.  This is not hard to see for each of the three cases.
\end{proof}

\begin{lemma}\label{lem:freechoice} 
	Let $B$ be a bag of $D$. If two unmarked vertices $x$ and $y$ are represented by a marked vertex $w$ in $B$, 
	then $\limb_D[B,x]$ is
  locally equivalent to $\limb_D[B,y]$.  
 \end{lemma}
\begin{proof}

Since $x$ and $y$ are represented by the same vertex $w$ of $B$, they are contained in the same component of $D\setminus V(B)$, say $T$.  
Let $v:=\zeta_t(D,B,T)$. 

If $B$ is a complete bag or a star bag having $w$ as a leaf,
then by the definition of limbs, 
  $\limb_D[B,x]=\limb_D[B,y]$.
  So, we may assume that $B$ is a star bag and $w$ is its center.
Since $v$ is linked to both $x$ and $y$ in $T$, 
by Lemma~\ref{lem:decompverlemma}, 
$T\wedge vx\wedge xy=T\wedge vy$.
So, we obtain that
$(T\wedge vx\setminus v) \wedge xy=T\wedge vx\wedge xy\setminus v=T\wedge vy\setminus v.$ Therefore
  $\limb_D[B,x]$ is locally equivalent to $\limb_D[B,y]$.
\end{proof}


 For a bag $B$ of $D$ and a component $T$ of $D\setminus V(B)$, we define $f_D(B,T)$ as the
 linear rank-width of $\limbhat_D[B,y]$ for some unmarked vertex $y\in V(T)$.  In fact, by Lemma~\ref{lem:freechoice}, $f_D(B,T)$ does not depend on the choice of $y$. 
 Furthermore, by the following proposition, it does not change when we replace $D$ with some decomposition locally equivalent to $D$.

 \begin{proposition}\label{prop:preservelrw} 
Let $B$ be a bag of $D$ and let $y$ be an unmarked vertex of $D$ represented a vertex $w$ in $B$.
 Let $x\in V(\origin{D})$.  
 If an unmarked vertex $y'$ is represented by $w$ in $D*x$, then $\limbhat_D[B,y]$ is locally equivalent to $\limbhat_{D*x}[(D*x)[V(B)],y']$. Therefore,
   $f_D(B,T)=f_{D*x}((D*x)[V(B)],T_x)$ where $T$ and $T_x$ are the components of $D\setminus V(B)$ and $(D*x)\setminus V(B)$ containing $y$, respectively.
   Moreover,  $\limbtil_D[B,y]$ and $\limbtil_{D*x}[(D*x)[V(B)],y']$ are locally equivalent as canonical decompositions. 
 \end{proposition}

Before proving it, let us recall the following by Geelen and Oum. 

\begin{lemma}[{Geelen and Oum \cite[Lemma 3.1]{GeelenO09}}]\label{lem:keylem}   Let $G$ be a graph and $x,y$ be two distinct vertices in $G$. Let $xw\in E(G*y)$ and $xz\in E(G)$.
  \begin{enumerate}
  \item If $xy\notin E(G)$, then $(G*y)\setminus x$, $(G*y*x)\setminus x$, and $(G*y)\wedge xw\setminus x$ are locally equivalent to $G\setminus x$, $G*x\setminus x$, and $G\wedge xz\setminus x$, respectively.
  \item If $xy\in E(G)$, then $(G*y)\setminus x$, $(G*y*x)\setminus x$, and $(G*y)\wedge xw\setminus x$ are locally equivalent to $G\setminus x$, $G\wedge xz\setminus x$, and $(G*x)\setminus x$, respectively.
  \end{enumerate}
\end{lemma}
%

\begin{proof}[Proof of Proposition \ref{prop:preservelrw}] 
Let $v:=\zeta_t(D,B,T)$ and $B':=(D*x)[V(B)]$.
Let $T$ and $T_x$ be the components of $D\setminus V(B)$ and $(D*x)\setminus V(B')$ containing $y$, respectively. 
Note that $V(T)=V(T_x)$.  

We claim that $\limbhat_D[B,y]$ is locally equivalent to $\limbhat_{D*x}[B',y']$ for some unmarked vertex $y'$ represented by $w$ in $D*x$.
We divide into cases depending on the type of the bag $B$  and whether $x\in V(T)$.

  \vskip 0.2cm
\noindent\emph{\textbf{Case 1.} $x\in V(T)$ and $x$ is not linked to $v$ in $T$.}

 Since $x$ is not linked to $v$ in $T$, applying a local complementation at $x$ does not change the bag $B$. 
 Thus, $B'=B$ and $vx\notin E(\origin{T})$.
 In this case, let $y':=y$.  

\begin{enumerate}
\item ($B$ is of type S and $w$ is a leaf of $B$.)  $\limb_D[B,y]=T\setminus v$ and $\limb_{D*x}[B',y']=T*x\setminus v$.
    Since $(T\setminus v)*x=T*x\setminus v$, 
    $\limb_D[B,y]$ and $\limb_{D*x}[B',y']$ are locally equivalent, and thus
    $\limbhat_D[B,y]$ and $\limbhat_{D*x}[B',y']$ are locally equivalent.
\item ($B$ is of type S and $w$ is the center of $B$.) $\limb_D[B,y]=T\wedge vy\setminus v$ and $\limb_{D*x}[B',y']=(T*x)\wedge vy\setminus v$, and we have
\[\limbhat_D[B,y]=\origin{T\wedge vy\setminus v}=\origin{T}\wedge vy\setminus v.\]
    Since $vx\notin E(\origin{T})$,  
 by Lemma~\ref{lem:keylem}, $\limbhat_D[B,y]$ is locally equivalent to
    \[\limbhat_{D*x}[B',y']=\origin{(T*x)\wedge vy\setminus v}=\origin{(T*x)}\wedge vy\setminus v.\]
    \item ($B$ is of type K.) $\limb_D[B,y]=T*v\setminus v$ and $\limb_{D*x}[B',y']=T*x*v\setminus v$, and we have
     \[\limbhat_D[B,y]=\origin{T* v\setminus v}=\origin{T}* v\setminus v.\] 
    Since $vx\notin E(\origin{T})$, 
  by Lemma~\ref{lem:keylem}, $\limbhat_D[B,y]$ is locally equivalent to 
   \[\limbhat_{D*x}[B',y']=\origin{T*x*v\setminus v}=\origin{T}*x*v\setminus v.\]
\end{enumerate}

%
%
%
%
%

    \vskip 0.2cm
\noindent\emph{\textbf{Case 2.} $x\in V(T)$ and $x$ is linked to $v$ in $T$.}

 Since $x$ is linked to $v$ in $T$, $vx\in E(\origin{T})$. Let $y':=x$ for this case.

\begin{enumerate}
\item ($B$ is of type S and $w$ is a leaf of $B$.) Applying a local complementation at $x$ does not change the type of the bag $B$. So, $\limb_D[B,y]=T\setminus v$ and $\limb_{D*x}[B',y']=T*x\setminus v$.  Since $(T\setminus v)*x=T*x\setminus v$, $\limbhat_D[B,y]$ and $\limbhat_{D*x}[B',y']$ are locally equivalent.
    \item ($B$ is of type S and $w$ is the center of $B$.)  Applying a local complementation at $x$ changes the bag $B$ into a bag of type K, and the component $T$ into $T*x$.  
    Thus, $\limb_D[B,y]=T\wedge vy\setminus v$ and $\limb_{D*x}[B',y']=(T*x)*v\setminus v$, and
    \[\limbhat_D[B,y]=\origin{T\wedge vy\setminus v}=\origin{T}\wedge vy\setminus v.\]
    Since $vx\in E(\origin{T})$, by Lemma~\ref{lem:keylem},
     $\limbhat_D[B,y]$  is locally equivalent to
    \[\limbhat_{D*x}[B',y']=\origin{(T*x)*v\setminus v}=\origin{T}*x*v\setminus v.\]
    \item ($B$ is of type K.) Applying a local complementation at $x$ changes the bag $B$ into a bag of type S whose center is $w$. 
    $\limb_D[B,y]=T*v\setminus v$ and $\limb_{D*x}[B',y']=T*x\wedge vx\setminus v$, and we have
    \[\limbhat_D[B,y]=\origin{T* v\setminus v}=\origin{T}* v\setminus v.\]
    Since $vx\in E(\origin{T})$, by Lemma~\ref{lem:keylem}, 
     $\limbhat_D[B,y]$ is locally equivalent to 
   \[\limbhat_{D*x}[B',y']=\origin{T*x\wedge vx\setminus v}=\origin{T}*x\wedge vx\setminus v.\]
\end{enumerate}

    \vskip 0.2cm
\noindent\emph{\textbf{Case 3.} $x\notin V(T)$.}

  If $x$ has no representative in the bag $B$, then applying a local complementation at $x$ does not change the bag $B$ and the component $T$.  Therefore, we may assume that $x$ is
  represented by some vertex in $B$, that is adjacent to $w$.  In this case, $v$ is still a representative of $y$ in $D*x$. Let $y':=y$.

\begin{enumerate}
\item ($B$ is of type S and $w$ is a leaf of $B$.) If the representative of $x$ in $B$ is a leaf of $B$, then 
it is not adjacent to $w$. 
   Thus, the representative of $x$ in $B$ is a center of $B$, and applying a local complementation at
    $x$ changes $B$ into a bag of type K, and $T$ into $T*v$.  
    We have $\limb_{D*x}[B',y']=(T*v)*v\setminus v=T\setminus v=\limb_D[B,y]$.
    \item ($B$ is of type S and $w$ is the center of $B$.) Since $w$ is the center of $B$, $x$ is represented by a leaf of the bag $B$. Applying a local complementation at $x$ does
    not change the bag $B$, but it changes $T$ into $T*v$.  So we have 
    $\limb_D[B,y]=T\wedge vy\setminus v$ and $\limb_{D*x}[B',y']=(T*v)\wedge vy\setminus v=T*y*v\setminus v$, and we have
    \[\limbhat_D[B,y]=\origin{T\wedge vy\setminus v}=\origin{T}\wedge vy\setminus v.\]
      Since $vy\in E(\origin{T})$, 
    by Lemma~\ref{lem:keylem},
     $\limbhat_D[B,y]$  is locally equivalent to
    \[\limbhat_{D*x}[B',y']=\origin{T*y*v\setminus v}=\origin{T}*y*v\setminus v.\]
    \item ($B$ is of type K.) After applying a local complementation at $x$ in $D$, $B$ becomes a star with a leaf $w$, and $T$ becomes
    $T*v$.  
    Therefore, we have $\limb_{D*x}[B',y']=T*v\setminus v=\limb_D[B,y]$.
\end{enumerate}

We conclude that $\limbhat_D[B,y]$ and $\limbhat_{D*x}[B',y']$ are locally equivalent, 
and by Lemma~\ref{lem:freechoice}, we have $f_D(B,T)=f_{D*x}(B',T_x)$. 
Also, by construction $\limbtil_D[B,y]$ and $\limbtil_{D*x}[B',y']$ are canonical decompositions of
$\limbhat_D[B,y]$ and $\limbhat_{D*x}[B',y']$, respectively.
 By Lemma~\ref{lem:localdecom},  we can conclude that $\limbtil_D[B,y]$ and $\limbtil_{D*x}[B',y']$ are locally equivalent as canonical decompositions.
 \end{proof}

The following lemma is useful to reduce cases in several proofs.

 \begin{lemma}\label{lem:fixedsd} 	
	Let $B_1$ and $B_2$ be two distinct bags of $D$ and for each $i\in \{1,2\}$, let $T_i$ be the components of $D\setminus V(B_i)$ such that $T_1$ contains the bag $B_2$ and $T_2$ contains the bag $B_1$. 
	Then there exists a canonical decomposition $D'$ locally equivalent to $D$ such that for each $i\in \{1,2\}$, $D'[V(B_i)]$ is a star and $\zeta_b(D,B_i,T_i)$ is a leaf of $D'[V(B_i)]$.
\end{lemma}

\begin{proof} Let $v_i:=\zeta_b(D,B_i,T_i)$ for $i=1,2$.  It is easy to make $B_1$ into a star bag having $v_1$ as a leaf by applying local complementations. 
We may assume that $v_1$ is a leaf of $B_1$ in $D$.  
If $v_2$ is a leaf of $B_2$, then we are done.  If $B_2$ is a complete bag, then choose an unmarked vertex $w_2$ of $D$ that is represented by a
  vertex of $B_2$ other than $v_2$.  Then applying a local complementation at $w_2$ makes $B_2$ into a star bag having $v_2$ as a leaf without changing $B_1$.  Therefore, we may assume that $v_2$ is
  the center of the star bag $B_2$.  If $B_1$ and $B_2$ are neighbor bags in $D$, then the marked edge connecting $B_1$ and $B_2$ is of type $S_pS_c$, contradicting to the assumption that $D$ is a
  canonical decomposition.  Thus, $B_1$ and $B_2$ are not neighbor bags in $D$.

	Let
	$T:=D[V(T_1)\cap V(T_2)]$ and $w_2:=\zeta_t(D, B_2, T_2)$.
	By the definition of a canonical decomposition, 
	$w_2$ is not a leaf of a star bag in $D$.
	Therefore,
	there exists an unmarked vertex $y\in V(T)$ of $D$ such that $y$ is linked to $w_2$ in $T$.
	Choose an unmarked vertex $y'$ of $D$ represented by $w_2$ in $D$.
	Since $y$ is linked to $y'$ and the alternating path from $y$ to $y'$ in $D$ passes through $B_2$ but not $B_1$,
	pivoting $yy'$ in $D$ makes $B_2$ into a star bag having $v_2$ as a leaf without changing $B_1$.
	Thus, each $v_i$ is a leaf of $(D\wedge yy')[V(B_i)]$ in $D\wedge yy'$, as required.
	\end{proof}

	We conclude the section with the following.

\begin{proposition}\label{prop:containing2} Let $B_1$ and $B_2$ be two distinct bags of $D$ and $T_1$ be a component of $D\setminus V(B_1)$
  not containing $B_2$, and $T_2$ be the component of $D\setminus V(B_2)$ containing $B_1$.  If $y_1\in V(T_1)$ and $y_2\in V(T_2)$ are two unmarked vertices of $D$ that are represented by some vertices in $B_1$ and $B_2$, respectively,  
  then $\limbhat_D[ B_1,y_1]$ is a vertex-minor of $\limbhat_D[B_2,y_2]$. Therefore $f_D(B_1,T_1)\le f_D(B_2,T_2)$.
\end{proposition}

   \begin{proof}
   Let $u_2:=\zeta_t(D,B_2, T_2)$ and $v_2:=\zeta_b(D,B_2, T_2)$.
   By Lemma~\ref{lem:fixedsd}, there exists a canonical decomposition $D'$ locally equivalent to $D$ such that $B_2$ is a star bag in $D'$ with a leaf $v_2$.
   For each $i\in \{1,2\}$, let 
$T_i':=D'[V(T_i)]$, $B_i':=D'[V(B_i)]$ and
  let $y_i'$ be an unmarked vertex of $D'$ represented by $\zeta_b(D', B_i', T_i')$.
  
  Since $v_2$ is a leaf of $B'_2$ in $D'$, 
  we have $\limb_{D'}[B'_2, y_2']=T'_2\setminus v_2$.
  Because $T'_1$ is a subgraph of $T'_2\setminus v_2$,
  we can easily observe that $\limbhat_{D'}[B'_1, y_1']$ is a vertex-minor of $\limbhat_{D'}[ B'_2, y_2']$.
  Since $\limb_D[B_i, y_i]$ is locally equivalent to $\limb_{D'}[B'_i, y_i']$ for each $i\in \{1,2\}$, $\limbhat_D[B_1, y_1]$ is a vertex-minor of $\limbhat_D[B_2, y_2]$.
   We conclude that $f_D(B_1,T_1)\le f_D(B_2,T_2)$. 
\end{proof}


\section{Characterizing the linear rank-width of distance-hereditary graphs}\label{sec:dh-lrw}

In this section, we prove the main structural result of this paper, which characterizes linear rank-width on distance-hereditary graphs. 
      
\begin{theorem}\label{thm:main}  
Let $k$ be a positive integer and let $D$ be the canonical decomposition of a connected distance-hereditary graph $G$.
Then $\lrw(G)\le k$ if and only if for each bag
  $B$ of $D$, $D$ has at most two components $T$ of $D\setminus V(B)$ such that $f_D(B,T)=k$, and every other component $T'$ of $D\setminus V(B)$ satisfies that $f_D(B,T')\le k-1$.
\end{theorem}

Let $D$ be the canonical decomposition of a connected distance-hereditary graph $G$, and we fix a positive integer $k$.
 For simpler arguments, we remove $D$ from the notation $f_D(B,T)$ in this section.  
We first prove the forward direction.


\begin{proof}[Proof of the forward direction of Theorem~\ref{thm:main}]
Suppose that there exists a bag $B$ of $D$ such that $D\setminus V(B)$ has at least
 three components $T$ which induce limbs $L$ where $\origin{L}$ has linear rank-width $k$.
 
 We claim that $\lrw(G)\ge k+1$.
We may assume that $D\setminus V(B)$ has exactly three components $T_1$, $T_2$ and $T_3$, where each component $T_i$ 
satisfies $f(B,T_i)=k$.  
	For each $1\le i\le 3$,
	let $w_i:=\zeta_t(D,B,T_i)$, and
	let $N_i$ be the set of the unmarked vertices of $T_i$ that are linked to $w_i$ in $T_i$.
	Choose a vertex $u_i$ in $N_i$ and let $D_i:=\limb_D[B,u_i]$.
	We remark that $N_i$ is exactly the set of the vertices in $V(\origin{D_i})$ that have a neighbor in $V(\origin{D})\setminus V(\origin{D_i})$.
Since removing a vertex from a graph does not increase the linear rank-width,
we may assume that $V(B)=\{\zeta_b(D,B,T_i) \mid 1\le i\le 3\}$.
Now, every unmarked vertex of $D$ is contained in one of $T_1$, $T_2$, and $T_3$.
Moreover, by Proposition~\ref{prop:preservelrw} and Lemmas \ref{lem:vm-rw} and \ref{lem:localdecom}, for any canonical decomposition $D'$ locally equivalent to $D$, we have $\lrw(\origin{D}) = \lrw(\origin{D'})$ and $f(B,T_i)$ does not change.  
So, we may assume that $B$ is a complete bag of $D$.

	We first claim that $D_2=(D*u_1)[V(T_2)\setminus w_2]$.
	Since $B$ is a complete bag, by the definition of limbs, 
	$D_2=T_2*w_2\setminus w_2$.
	Since $u_1$ is linked to $w_1$ in $T_1$ and there is an alternating path from $w_1$ to $w_2$ in $D$,
	by concatenating alternating paths
	it is easy to see that 
	$(D*u_1)[V(T_2)\setminus w_2] = T_2*w_2\setminus w_2=D_2$,
	as claimed. 

  Towards a contradiction, suppose that $\origin{D}$ has a linear layout $L$ of width $k$.  Let $a$ and $b$ be the first vertex and the last vertex of $L$, respectively.  
  Since $B$ has no unmarked vertices, 
  without loss of generality, we may assume
  that $a,b\in V(\origin{D_1})\cup V(\origin{D_3})$.  With this assumption, we claim that $\origin{D_2}$ has linear rank-width at most $k-1$.

  Let $v\in V(\origin{D_2})$ and $S_v:=\{x\in V(\origin{D})\mid x\le_L v\}$
  and $T_v:=V(\origin{D})\setminus S_v$.
  Since $v$ is arbitrary, it is sufficient to show that
  $\cutrk_{\origin{D_2}}(S_v\cap V(\origin{D_2})) \le k-1$.
  
      We divide into three cases. 
      We first check two cases that are 
    	(1) ($N_1\cap S_v\neq\emptyset$ and $N_3\cap T_v\neq \emptyset$),
    	and (2) ($N_1\cap T_v\neq\emptyset$ and $N_3\cap S_v\neq \emptyset$).
      If both of them are not satisfied, then 
      we can easily deduce that $N_1\cup N_3\subseteq S_v$ or $N_1\cup N_3\subseteq T_v$.
          
    \vskip 0.2cm
\noindent\emph{\textbf{Case 1.} $N_1\cap S_v\neq\emptyset$ and $N_3\cap T_v\neq \emptyset$.}

      Let $x_1\in N_1\cap S_v$ and $x_3\in N_3\cap T_v$.
  We claim that 
  \[\cutrk_{\origin{D_2}}(S_v\cap V(\origin{D_2}))=
   \cutrk_{(\origin{D})[V(\origin{D_2}) \cup \{x_1,x_3\}]}((S_v\cap V(\origin{D_2}))\cup \{x_1\})-1.
  \]
  Because $\cutrk_{(\origin{D})[V(\origin{D_2})\cup \{x_1,x_3\}]}((S_v\cap V(\origin{D_2}))\cup \{x_1\})
   \le \cutrk_{\origin{D}}(S_v)\le k,$
  the claim implies that 
  $\cutrk_{\origin{D_2}}(S_v\cap V(\origin{D_2})) \le k-1$.
  
  Note that $x_1$ and $x_3$ have the same neighbors in $(\origin{D})[V(\origin{D_2}) \cup \{x_1,x_3\}]$ because $B$ is a complete bag.  Since
  $x_1$ is adjacent to $x_3$ in $(\origin{D})[V(\origin{D_2}) \cup \{x_1,x_3\}]$, $x_3$ becomes a leaf in $(\origin{D})[V(\origin{D_2}) \cup \{x_1,x_3\}]*x_1$ whose neighbor is $x_1$.
  Since $(D*x_1)[V(T_2)\setminus w_2]=D_2$, we have
  \[(\origin{D})[V(\origin{D_2}) \cup \{x_1,x_3\}]*x_1\setminus x_1\setminus x_3=
  (\origin{D}*x_1)[V(\origin{D_2})]=\origin{D_2}.\]
  Therefore,
   \begin{align*}
    & \cutrk_{(\origin{D})[V(\origin{D_2}) \cup \{x_1,x_3\}]}((S_v\cap V(\origin{D_2}))\cup \{x_1\}) \\
    &=\cutrk_{(\origin{D})[V(\origin{D_2}) \cup \{x_1,x_3\}]*x_1}((S_v\cap V(\origin{D_2}))\cup \{x_1\}) \\
    &=\rank \left(\begin{blockarray}{ccccccccc} &x_3&T_v\cap V(\origin{D_2}) \\
        \begin{block}{c(c|ccccccc)} x_1&1&* \\ \cline{2-5} 
        S_v\cap V(\origin{D_2}) &0 &* \\
        \end{block}
      \end{blockarray}\right) \\ 
      &=\rank \left(\begin{blockarray}{ccccccccc} &x_3&T_v\cap V(\origin{D_2}) \\
        \begin{block}{c(c|ccccccc)} x_1&1&0 \\ \cline{2-5} 
        S_v\cap V(\origin{D_2}) &0 &* \\
        \end{block}
      \end{blockarray}\right) \\
    &=\cutrk_{(\origin{D})[V(\origin{D_2})\cup\{x_1, x_3\}]*x_1\setminus x_1\setminus x_3}(S_v\cap V(\origin{D_2})) + 1\\
    &=\cutrk_{(\origin{D_2})}(S_v\cap V(\origin{D_2})) +1,
  \end{align*} 
  as claimed. 
  
    \vskip 0.2cm
\noindent\emph{\textbf{Case 2.} $N_1\cap T_v\neq \emptyset$ and $N_3\cap S_v\neq\emptyset$.}

   In the same way as \emph{\textbf{Case 1}}, we can prove $\cutrk_{\origin{D_2}}(S_v\cap V(\origin{D_2})) \le k-1$.

    \vskip 0.2cm
\noindent\emph{\textbf{Case 3.} $N_1\cup N_3\subseteq S_v$ or $N_1\cup N_3\subseteq T_v$.}

  We can assume without loss of generality that $N_1\cup N_3\subseteq S_v$ because the case when $N_1\cup N_3\subseteq T_v$ is similar.
  Since $a,b\in V(\origin{D_1})\cup V(\origin{D_3})$ and the
  graph $(\origin{D})[V(\origin{D_1})\cup V(\origin{D_3})]$ is connected,
  there exist vertices $s\in S_v\cap (V(\origin{D_1})\cup V(\origin{D_3}))$ and $t\in T_v\cap (V(\origin{D_1})\cup V(\origin{D_3}))$ such that
  \begin{enumerate}
  \item $st\in E(\origin{D})$, 
  \item $t$ has no neighbors in $N_2$.
  \end{enumerate}
 We have
   \begin{align*}
   \cutrk_{\origin{D}}(S_v) 
   &\ge\rank \left(\begin{blockarray}{ccccccccc} 
   &t&T_v\cap V(\origin{D_2}) \\
        \begin{block}{c(c|ccccccc)} 
        s&1&* \\ \cline{2-5} 
       S_v\cap V(\origin{D_2}) &0 &* \\
        \end{block}
      \end{blockarray}\right) \\ 
      &=\rank \left(\begin{blockarray}{ccccccccc} 
   &t&T_v\cap V(\origin{D_2}) \\
        \begin{block}{c(c|ccccccc)} 
        s&1&0 \\ \cline{2-5} 
       S_v\cap V(\origin{D_2}) &0 &* \\
        \end{block}
      \end{blockarray}\right) \\ 
    &=\cutrk_{\origin{D_2}}(S_v\cap V(\origin{D_2})) +1.
  \end{align*} 
  Therefore, we conclude $\cutrk_{\origin{D_2}}(S_v\cap V(\origin{D_2}))\le k-1$. 

  Thus, $\origin{D_2}$ has linear rank-width at most $k-1$, which yields a contradiction.  
\end{proof}


To prove the converse direction, we use the following lemmas.
For two linear layouts $(x_1, \ldots, x_n)$, $(y_1, \ldots, y_m)$, 
we define 
\[(x_1, \ldots, x_n)\oplus (y_1, \ldots, y_m):=(x_1, \ldots, x_n, y_1, \ldots, y_m).\]

\begin{lemma}\label{lem:nomarkedcenter} Let $B$ be a bag of $D$ of type S  with two unmarked vertices $x$ and $y$ such that $x$ is the center and $y$ is a leaf of $B$.    
  If for every component $T$ of $D\setminus V(B)$, $f(B,T)\le k-1$,
  then the graph $\origin{D}$ has a linear layout of width at most $k$ whose first and last vertices are $x$ and $y$, respectively.
  \end{lemma}

\begin{proof}
  Let $T_1, T_2, \ldots, T_{\ell}$ be the components of $D\setminus V(B)$ and for each $1\le i\le \ell$, let $w_i:=\zeta_t(D,B, T_i)$ and let $y_i$ be a vertex in $T_i$ represented by a vertex of $B$.  Since each $w_i$ is
  adjacent to a leaf of $B$, $T_i\setminus w_i$ is the limb of $D$ with respect to $B$ and $y_i$.
  Let $A:=V(B)\setminus (\bigcup_{1\le j\le \ell} \{\zeta_b(D, B, T_i)\})\setminus \{x,y\}$, and let $L_A$ be a sequence of $A$.		
 
 Suppose that for every component $T$ of $D\setminus V(B)$, $f(B,T)\le k-1$.  
  For each $1\le i\le \ell$, let $L_i$ be a linear layout of $\origin{T_i\setminus w_i}$ of width at most $k-1$.  We claim that
  \[L:=(x)\oplus L_1\oplus L_2 \oplus \cdots \oplus L_{\ell} \oplus L_A \oplus (y)\] is a linear layout of $\origin{D}$ of width at most $k$.  It is sufficient to prove that for every $w\in V(\origin{D})\setminus \{x,y\}$,
  $\cutrk_{\origin{D}}(\{v\mid v\le_L w\})\le k$.
    
  Let $w\in V(\origin{D})\setminus (A\cup \{x,y\})$, and let $S_w:=\{v:v\le_L w\}$ and $T_w:=V(\origin{D})\setminus S_w$.  
  Let $j$ be the integer such that $L_j$ contains $w$. Then
  \begin{align*} 
  &\cutrk_{\origin{D}}(S_w)  \\
  &=\rank \left(\begin{blockarray}{ccccccccc} &y&A&T_w\cap V(\origin{T_j}) & T_w\setminus \{y\} \setminus A\setminus V(\origin{T_j}) \\
        \begin{block}{c(c|c|c|ccccc)} x&1&1&* &* \\ \cline{2-5} S_w\cap V(\origin{T_j})&0&0 &* &0 \\ \cline{2-5} S_w\setminus \{x\} \setminus V(\origin{T_j})&0&0&0 &0 \\
        \end{block}
      \end{blockarray}\right) \\ &=\rank \left(\begin{blockarray}{ccccccccc} &y&A&T_w\cap V(\origin{T_j}) & T_w\setminus \{y\} \setminus A\setminus V(\origin{T_j}) \\
        \begin{block}{c(c|c|c|ccccc)} x&1&0&0 &0 \\ \cline{2-5} S_w\cap V(\origin{T_j})&0 &0&* &0 \\ \cline{2-5} S_w\setminus \{x\}\setminus V(\origin{T_j})&0 &0&0 &0 \\
        \end{block}
      \end{blockarray}\right) \\ &=\cutrk_{\origin{T_j\setminus w_j}}(S_w\cap V(\origin{T_j})) +1 \le (k-1)+1 = k.
  \end{align*} 
  If $w\in A$, then it is easy to show that $\cutrk_{\origin{D}}(\{v\mid v\le_L w\})\le 1$.  
  Therefore, $L$ is a linear layout of $\origin{D}$ of width $k$ whose first and last vertices are $x$ and $y$, respectively.
\end{proof}

\begin{lemma}\label{lem:equiv} Let $B$ be a bag of $D$ with two unmarked vertices $x$ and $y$.
	If for every component $T$ of $D\setminus V(B)$, $f(B,T)\le k-1$,
	then the graph $\origin{D}$ has a linear layout of width at most $k$ whose first and last vertices are $x$ and $y$, respectively.
\end{lemma}

\begin{proof}
  Suppose that $f(B,T)\le k-1$ for every component $T$ of $D\setminus V(B)$. 
  We obtain a decomposition $D'$ from $D$ as follows:
  \begin{itemize}
  \item If $B$ is a complete graph, then let $D':=D*x$. 
  \item If $B$ is a star whose center is $x$, then let $D':=D$.
  \item If $y$ is the center of $B$, then let $D':=D\wedge xy$.
  \item Otherwise let $D':=D\wedge xz$ where $z$ is an unmarked vertex represented by the center of $B$. 
  \end{itemize} 
  It is
  clear that $D'[V(B)]$ is a star whose center is $x$. By Proposition~\ref{prop:preservelrw}, for each component $T$ of $D\setminus V(B)$, $f(B,T)=f_{D'}(D'[V(B)],D'[V(T)])$.
  Thus, by Lemma~\ref{lem:nomarkedcenter}, $\origin{D'}$ has a linear layout of width at most $k$ whose first and last vertices are $x$ and $y$, respectively.
  Since $\origin{D'}$ is locally equivalent to $\origin{D}$, we conclude that $\origin{D}$ also has such a linear layout.
\end{proof}


\begin{lemma}\label{lem:sdpath} 
If 
\begin{enumerate}
\item for each bag $B$ of $D$, there are at most two components $T$
  of $D\setminus V(B)$ satisfying $f(B,T)=k$, and 
  \item for every other component $T'$ of $D\setminus V(B)$, $f(B,T')\le k-1$, and 
  \item $P$ is the set of nodes $v$ in $T_D$ such that exactly two components $T$ of
  $D\setminus V(\bag{v}{D})$ satisfy $f(\bag{v}{D},T)=k$, 
  \end{enumerate} then either $P=\emptyset$ or $T_D[P]$ is a path.
\end{lemma}

\begin{proof} 
  Suppose that $P\neq \emptyset$.  If $P$ has two distinct vertices $v_1$ and $v_2$, then there exists a component $T_1$ of $D\setminus V(\bag{D}{v_1})$ not containing $V(\bag{D}{v_2})$ such that
  $f(\bag{D}{v_1},T_1)=k$, and there exists a component $T_2$ of $D\setminus V(\bag{D}{v_2})$ not containing $V(\bag{D}{v_1})$ such that $f(\bag{D}{v_2},T_2)=k$.  By
  Proposition~\ref{prop:containing2}, for every node $v$ on the path from $v_1$ to $v_2$ in $T_D$, $v$ must be contained in $P$.  So $P$ induces a tree in $T_D$.

  Suppose now that $P$ contains a node $v$ having three neighbor bags $v_1, v_2,$ and $v_3$ in $P$.  Then, again by Proposition~\ref{prop:containing2}, $D$ must have three components $T$ of
  $D\setminus V(\bag{D}{v})$ such that $f(\bag{D}{v},T)=k$, which contradicts the assumption.  Therefore, $P$ induces a path in $T_D$.
\end{proof}

\begin{lemma}\label{lem:sdpath2} 
If 
\begin{enumerate}
\item for each bag $B$ of $D$, there are at most two
  components $T$ of $D\setminus V(B)$ satisfying $f(B,T)=k$, and 
  \item $f(B,T')\le k-1$ for all the other components $T'$ of $D\setminus V(B)$,  
  \end{enumerate}
  then $T_D$ has a path $P$ such that for each node $v$ in $P$
  and each component $T$ of $D\setminus V(\bag{v}{D})$ not containing a bag $\bag{w}{D}$ with $w\in V(P)$, $f(B,T)\le k-1$.
\end{lemma}

\begin{proof} Let $P'$ be the set of nodes $v$ in $T_D$ such that exactly two components $T$ of $D\setminus V(\bag{D}{v})$ satisfy $f(\bag{D}{v},T)=k$. By Lemma~\ref{lem:sdpath}, either $P'=\emptyset$
  or $T_D[P']$ is a path.
  
  We first assume that $P'\neq \emptyset$.  Let $T_D[P']=v_1v_2 \cdots v_n$, and for each $1\le i\le n$, let $B_i:=\bag{D}{v_i}$.  By the definition, there exists a component $T_1$ of $D\setminus V(B_1)$ such that $T_1$
  does not contain a bag $\bag{D}{w}$ with $w\in V(P')$ and $f(B_1,T_1)=k$.  Let $v_0$ be the node of $T_D$ such that $\bag{D}{v_0}$ is the bag of $T_1$ that is the neighbor bag of $B_1$ in $D$.  Similarly, there exists a
  component $T_n$ of $D\setminus V(B_n)$ such that $T_n$ does not contain a  bag $\bag{D}{w}$ with $w\in V(P')$ and $f(B_n,T_n)=k$.  Let $v_{n+1}$ be the node of $T_D$ such that $\bag{D}{v_{n+1}}$ is the bag of $T_n$ that is
  the neighbor bag of $B_n$ in $D$.  Then $P:=v_0v_1v_2 \cdots v_nv_{n+1}$ is the required path.
  
  Now we assume that $P'=\emptyset$.  We choose a node $v_0$ in $T_D$ and let $B_0:=\bag{D}{v_0}$.  If $D$ has no component $T$ of $D\setminus V(B_0)$ such that $f(B_0,T)=k$, then $P:=v_0$ satisfies
  the condition.  If not, we take a maximal path $P:=v_0v_1 \cdots v_{n+1}$ in $T_D$ such that (with $B_i:=\bag{D}{v_i}$)
  \begin{itemize}[label={--}]
  \item for each $0\le i\le n$, $D\setminus V(B_i)$ has one component $T_i$ such that $f(B_i,T_i)=k$, and $B_{i+1}$ is the bag of $T_i$ that is the neighbor bag of $B_i$ in $D$.
  \end{itemize} 

  By the maximality of $P$, $P$ is a path in $T_D$ such that for each node $v$ of $P$ and a component $T$ of $D\setminus V(\bag{D}{v})$ not containing a bag $\bag{D}{w}$ with $w\in V(P)$, $f(\bag{D}{v},T)\le
  k-1$. 
\end{proof}

We are now ready to prove the converse direction of the proof of Theorem~\ref{thm:main}.

\begin{proof}[Proof of the backward direction of Theorem~\ref{thm:main}]
Suppose that for each bag $B$ of $D$, at most two components $T$ of $D\setminus V(B)$ induce limbs $L$ where $\origin{L}$ has linear rank-width exactly $k$, and all other component $T'$ of $D\setminus V(B)$ induce limbs $L'$ where $\origin{L'}$ has linear rank-width at most $k-1$.  
We claim that $\lrw(G)\le k$.

 Let $P:=v_0v_1 \cdots v_nv_{n+1}$ be the path in $T_D$ such that 
 \begin{itemize}
 \item for each node $v$ in $P$ and a component $T$ of $D\setminus V(\bag{D}{v})$ not containing a bag $\bag{D}{w}$ with
  $w\in P$, $f(\bag{D}{v},T)\le k-1$ (such a path exists by Lemma \ref{lem:sdpath2}).  
  \end{itemize}
  For each $0\leq i \leq n+1$, let $B_i:=\bag{D}{v_i}$. 
   If $P$ consists of one vertex, then by Lemma~\ref{lem:equiv},
  $\lrw(G)=\lrw(\origin{D})\le k$. Thus, we may assume that $n\ge 0$.
 
  By adding unmarked vertices in $B_0$ and $B_{n+1}$ if necessary, we assume that $B_0$ and $B_{n+1}$ have unmarked vertices $a_0$ and $b_{n+1}$ in $D$, respectively.
  
   For each $0\le i\le n$, let $b_i$ be a marked vertex of $B_i$ and let $a_{i+1}$ be a marked vertex $B_{i+1}$ such that 
 $b_ia_{i+1}$ is the marked edge connecting $B_i$ and $B_{i+1}$. 
  Let $D_0$ be the component of $D\setminus V(B_1)$ containing the bag $B_0$.
  Let $D_{n+1}$ be the component of $D\setminus V(B_n)$ containing the bag $B_{n+1}$.
 For each $1\le i\le n$,
 let $D_i$ be the component of $D\setminus (V(B_{i-1})\cup V(B_{i+1}))$ containing the bag $B_i$. 
 Notice that the vertices $a_i$ and $b_i$ are unmarked vertices in $D_i$.

  Since every component $T$ of $D_i\setminus V(B_i)$ satisfies that $f_{D_i}(B_i,T) \leq k-1$, by Lemma~\ref{lem:equiv}, $\origin{D_i}$ has a linear layout $L_i'$ of width $k$ whose first and last vertices are $a_i$ and $b_i$, respectively.  For each $1\leq i \leq n$, let $L_i$ be the linear layout obtained from $L_i'$ by removing $a_i$ and $b_i$.  Let $L_0$ and $L_{n+1}$
  be obtained from $L_0'$ and $L_{n+1}'$ by removing $b_0$ and $a_{n+1}$, respectively. Then we can easily check that
  $L:= L_0\oplus L_1\oplus \cdots \oplus L_{n+1}$ is a linear layout of $\origin{D}$ having width at most $k$.  Therefore $\lrw(\origin{D})\le k$.
\end{proof}




\section{Canonical limbs}\label{sec:lemmaonCL}
    

We investigate useful properties of canonical limbs which will be used to design our algorithm.
Note that for recursively taking limbs, we need to transform an obtained limb into a canonical limb because limbs are only defined on canonical decompositions.
Let $D$ be the canonical decomposition of a connected distance-hereditary graph.

\begin{proposition}\label{prop:order} 
Let $B_1$ and $B_2$ be two distinct bags of $D$ and 
  for each $i\in \{1,2\}$, let $T_i$ be the component of $D\setminus V(B_i)$, $w_i:=\zeta_b(D, B_i, T_i)$ and $y_i$ be an unmarked vertex of $D$ represented by $w_i$ such that
  \begin{itemize}
  \item
  $T_1$ contains the bag $B_2$ and $T_2$ contains the bag $B_1$, and
\item $V(B_{1})$ induces a bag in $\limbtil_D[B_2,y_2]$, and $V(B_{2})$ induces a bag in $\limbtil_D[B_1,y_1]$.
\end{itemize}
  We define that
  \begin{itemize}
  \item $B_{1}':=(\limbtil_D[B_2,y_2])[V(B_{1})]$,
  \item $B_{2}':=(\limbtil_D[B_1,y_1])[V(B_{2})]$, 
  \item $y_{1}'$ is an unmarked vertex of $\limbtil_D[B_2,y_2]$  represented by $w_{1}$, and
  \item $y_{2}'$ is an unmarked vertex of $\limbtil_D[B_1,y_1]$ represented by $w_{2}$.  
  \end{itemize}
  Then
  $\limbtil_{\limbtil_D[B_1,y_1]}[B_2', y_2']$ is locally equivalent to $\limbtil_{\limbtil_D[B_2,y_2]}[B_1', y_1']$.
\end{proposition}

\begin{proof}
	 For each $i\in \{1,2\}$, let $v_i:=\zeta_t(D, B_i, T_i)$.
	By Lemma~\ref{lem:fixedsd}, 
 	  there exists a canonical decomposition $D'$ locally equivalent to $D$ such that for each $i\in \{1,2\}$, $w_i$ is a leaf of $D'[V(B_i)]$ in $D'$.
 	  	 For each $i\in \{1,2\}$, let $P_i:=D'[V(B_i)]$, $T_i':=D'[V(T_i)]$, and $z_i$ be an unmarked vertex of $D'$ represented by $w_i$.  	
	 We define that 
  \begin{itemize}
  \item $T':=D'[V(T_1')\cap V(T_2')]$,
   \item $P_{1}':=(\limbtil_{D'}[P_2,z_2])[V(P_{1})]$, 
  \item $P_{2}':=(\limbtil_{D'}[P_1,z_1])[V(P_{2})]$, 
  \item let $z_{1}'$ be an unmarked vertex of $\limbtil_{D'}[P_2,z_2]$ represented  by $w_{1}$,
    \item let $z_{2}'$ be an unmarked vertex of $\limbtil_{D'}[P_1,z_1]$ represented by $w_{2}$. 
  \end{itemize}

 	 Since $D$ is locally equivalent to $D'$, 
 	 by Proposition~\ref{prop:preservelrw},
 	 $\limbtil_{D}[B_1,y_1]$ is locally equivalent to 
 	 $\limbtil_{D'}[P_1,z_1]$.
 	 Again, since $\limbtil_{D}[B_1,y_1]$ is locally equivalent to 
 	 $\limbtil_{D'}[P_1,z_1]$,
 	 by Proposition~\ref{prop:preservelrw},
 	 \[\text{$\limbtil_{\limbtil_{D}[B_1,y_1]}[B_2', y_2']$ is locally equivalent to 
 	 $\limbtil_{\limbtil_{D'}[P_1,z_1]}[P_2', z_2']$.}\]
 	 Similarly, we obtain that
 	 \[\text{$\limbtil_{\limbtil_{D}[B_2,y_2]} [B_1', y_1']$ is locally equivalent to 
 	 $\limbtil_{\limbtil_{D'}[P_2,z_2]}[P_1', z_1']$.}\]
	Since each $v_i$ is a leaf of $P_i$ in $D'$,
 	  	  	  \[\text{$\limb_{\limb_{D'}[P_1,z_1]}[P_2', z_2']=T'\setminus v_1\setminus v_2=\limb_{\limb_{D'}[P_2,z_2]}[P_1', z_1']$,}\]
		and it implies that
 	  \[\text{$\limbtil_{\limbtil_{D'}[P_1,z_1]}[P_2', z_2']=\limbtil_{\limbtil_{D'}[P_2,z_2]}[P_1', z_1']$.}\]	
			Therefore,
	$\limbtil_{\limbtil_D[B_1,y_1]}[B_2', y_2']$ is locally equivalent to 
	$\limbtil_{\limbtil_D[B_2,y_2]}[B_1', y_1']$.
	   \end{proof}

\begin{proposition}\label{prop:order2}
Let $B_1$ and $B_2$ be two distinct bags of $D$. 
 Let $T_1$ be a component of $D\setminus V(B_1)$ that does not contain $B_2$ and $T_2$ be the component of $D\setminus V(B_2)$
  containing the bag $B_1$.  
  For $i\in \{1,2\}$, let $w_i:=\zeta_b(D, B_i, T_i)$, and $y_i$ be an unmarked vertex of $D$ represented by $w_i$.
    If $V(B_1)$ induces a bag $B_1'$ of $\limbtil_{D}[B_2,y_2]$, then $\limbtil_{D}[B_1,y_1]$ is locally equivalent to $\limbtil_{\limbtil_{D}[B_2,y_2]}[B_1', y_1']$, 
  where $y_1'$ is an unmarked vertex of $\limbtil_D[B_2,y_2]$ represented by $w_1$.  
\end{proposition}

\begin{proof}
	Suppose $V(B_1)$ induces a bag $B_1'$ of $\limbtil_D[B_2,y_2]$ and 
	$y_1'$ is an unmarked vertex represented in $\limbtil_D[B_2,y_2]$ by $w_1$. 
	By Lemma~\ref{lem:fixedsd}, there exists a canonical decomposition $D'$ locally equivalent to $D$ such that
	$w_2$ is a leaf of a star bag $D'[V(B_2)]$.
	We define 
	\begin{itemize}
 	\item $P_1:=D'[V(B_1)]$,
       \item $P_2:=D'[V(B_2)]$,
	\item for each $i\in \{1,2\}$, $z_i$ is an unmarked vertex of $D'$ represented by $w_i$,
	\item $P_1':=(\limbtil_{D'}[P_2,z_2])[V(B_1)]$, and
	\item $z_1'$ is an unmarked vertex of $\limbtil_{D'}[P_2,z_2]$ represented by $w_1$.
	\end{itemize}

 	Since $D$ is locally equivalent to $D'$,
 	by Proposition~\ref{prop:preservelrw},
 	$\limbtil_D[B_1,y_1]$ is locally equivalent to $\limbtil_{D'}[P_1,z_1]$.
 	Similarly, we obtain that
 	$\limbtil_{D}[B_2,y_2]$ is locally equivalent to $\limbtil_{D'}[P_2,z_2]$.
 	Since $\limbtil_{D}[B_2,y_2]$ is locally equivalent to $\limbtil_{D'}[P_2,z_2]$,
 	by Proposition~\ref{prop:preservelrw},
 	\[\text{$\limbtil_{\limbtil_{D}[B_2,y_2]} [B_1', y_1']$ is locally equivalent to
 	$\limbtil_{\limbtil_{D'}[P_2,z_2]}[P_1', z_1']$.}\]
 	Since $w_2$ is a leaf of $P_2$ in $D'$,
 	 $\limbtil_{D'}[P_1,z_1]=\limbtil_{\limbtil_{D'}[P_2,z_2]}[P_1', z_1']$,
 	 and therefore, 
 	 $\limbtil_{D}[B_1,y_1]$ is locally equivalent to
 	 $\limbtil_{\limbtil_{D}[B_2,y_2]} [B_1', y_1'],$ as required.
 \end{proof}

\section{Computing the linear rank-width of distance-hereditary graphs}\label{sec:computing-lrw}

We describe an algorithm to compute the linear rank-width of distance-hereditary graphs.  Since the linear rank-width of a graph is the maximum linear rank-width over all its connected components, we
will focus on connected distance-hereditary graphs.

\begin{theorem}\label{thm:mainalg} 
	The linear rank-width of every connected distance-hereditary graph with $n$ vertices can be computed in
	time $\cO (n^2\cdot \log_2 n)$. Moreover, 
a linear layout of the graph witnessing the linear rank-width can be computed with the same time complexity.
\end{theorem}


The main idea consists of rooting the canonical decomposition $D$ of a connected
distance-hereditary graph and associating each bag $B$ of $D$ with a canonical limb $\limbtil_D[B',y]$ where $B'$ is the parent of $B$ and $y$ is an unmarked vertex in some descendant bag of $B$, and computing the linear rank-width of $\limbhat_D[B',y]$. 
Following Theorem~\ref{thm:main}, in order to compute the linear rank-width of $\limbhat_D[B',y]$, 
we need to check the linear rank-width of proper limbs obtained from $\limbtil_D[B',y]$ by removing some bags of $\limbtil_D[B', y]$.
Basically, we need to take canonical limbs recursively from this reason.
In contrast to the case of forests for computing path-width, the associated canonical
limbs here are not necessarily sub-decompositions of the original decomposition, 
and thus, it is not at all trivial how to store values to use in the next steps. 	
The crucial point of achieving our running time is to overcome this problem using the results in Section~\ref{sec:lemmaonCL}.

    \vskip 0.2cm
\noindent{\textbf{Rooted decomposition trees.} We define the notion of \emph{rooted decomposition trees}.  A decomposition tree is \emph{rooted} if we distinguish either a node or an edge and call it the \emph{root} of the tree. Let $T$ be a rooted
decomposition tree with the root $r$. A node $v$ is a descendant of a node $v'$ if $v'$ is in the unique path from the root to $v$, and  when $r$ is a marked edge, this path contains both end vertices of $r$. If $v$ is a descendant of $v'$ and $v$ and $v'$ are adjacent, then we call $v$ a
\emph{child} of $v'$ and $v'$ the \emph{parent} of $v$. Observe from the definition of descendants that if $r=vv'$, then $v$ is the parent of $v'$ and also $v'$ is the parent of $v$.  We allow this tricky part for a technical reason. A node in $T$ is called a \emph{non-root node} if it is not the root node. 

Two nodes $v$ and
$v'$ are called \emph{comparable} if one node is a descendant of the other one. Otherwise, they are called \emph{incomparable}. Recall that for each node $v$ of $T$ and each canonical decomposition
$D$ with $T$ as its decomposition tree we write $\bag{D}{v}$ to denote the bag of $D$ with which it is in correspondence. For convenience, let $\pbag{D}{v}:=\bag{D}{v'}$ with $v'$ the parent of $v$.

Let $D$ be the canonical decomposition of a connected distance-hereditary graph $G$ and let $T$ be its decomposition tree rooted at $r$. 
Let $B:=\bag{D}{v}$ for some non-root node $v$ of $T$, and let $y$ be an unmarked vertex of $D$ that is represented by a vertex of $B$.
We define the root of the decomposition tree $\widetilde{T}$ of $\limbtil_D[B,y]$ as follows.
We assume that $\widetilde{T}$ is obtained from $T$ by removing $v$, and possibly adding an edge or identifying two nodes following the definition of canonical limbs.
If two comparable nodes $w$ and $w'$ with $w$ the parent of $w'$ are identified, then let $w$ be the identified node.
Otherwise, we give a new label for the identified node.

%

\begin{enumerate}
\item If $r$ exists in $\widetilde{T}$, then we assign $r$ as the root of $\widetilde{T}$. In the other cases, 
we can observe that either 
\begin{itemize}
\item $r$ is the root node and $\bag{D}{r}$ is
  removed when taking the canonical limb or 
  \item $r$ is the root edge, and a bag $\bag{D}{r'}$ is removed where $r'$ is a node incident with the root edge, when taking the canonical limb.
  \end{itemize}
\item If the removed node has one neighbor in $T\setminus r$, then we assign this neighbor as the root of $\widetilde{T}$.
\item If the removed node has two neighbors in $T\setminus r$ and they are linked by a new edge in $\widetilde{T}$, then we assign the new edge as the root of $\widetilde{T}$.
\item If the removed node has two neighbors in $T\setminus r$ and they are identified in $\widetilde{T}$, then we assign the new node as the root of $\widetilde{T}$.
\end{enumerate}

The following observation is easy to check from the definition of rooted decomposition trees of canonical limbs.

\begin{fact}\label{fact:non-root} 
If $w$ is a non-root node of the rooted decomposition tree $\widetilde{T}$ of a canonical
  limb $\limbtil_{D}[B,y]$, then $w$ is also a non-root node of $T$ with the property that $V(\bag{D}{w})=V(\bag{\limbtil_{D}[B,y]}{w})$.
\end{fact}


\vskip 0.2cm
\noindent{\textbf{$k$-critical nodes.} 
For every non-root node $v$ of $T$ with the parent node $v'$, we define that 
\begin{itemize}
\item $\cmp_1[D, v]$ is the component of $D\setminus V(\bag{D}{v'})$ containing $\bag{D}{v}$,
\item $\cmp_2[D, v]$ is the component of $D\setminus V(\bag{D}{v})$ containing $\bag{D}{v'}$,
\item $f_1(D,v):=f_D(\pbag{D}{v},\cmp_1[D,v])$, 
\item $f_2(D,v):=f_D(\bag{D}{v},\cmp_2[D,v])$,
\item $\zeta_1(D, v):=\zeta_b(D, \bag{D}{v'}, \cmp_1[D, v])$, and
\item $\zeta_2(D, v):=\zeta_b(D, \bag{D}{v}, \cmp_2[D, v])$.
\end{itemize}
A node $v$ of $T$ is called \emph{$k$-critical} if $f_1(D,v)=k$ and $v$ has two children $v_1$ and $v_2$ such that $f_1(D,v_1)=f_1(D,v_2)=k$.


\medskip

From now on, we define some sequences of canonical limbs, which will be taken sequentially in our algorithm.
We recall that $\lrw(G)\le \log_2\abs{V(G)}$ by Theorem~\ref{thm:rankwidth1} and Lemma~\ref{lem:lrwtrivialbound}.
For convenience, let 
\[\eta:=\lfloor \log_2\abs{V(G)}\rfloor.\]
For each non-root node $v$ of $T$, we define recursively the following. 
We first choose an unmarked vertex $y$ of $D$ represented by $\zeta_1(D, v)$, and
\begin{itemize}
\item let $D^{v}_{\maxnum}$ be any canonical limb $\limbtil_{D}[\pbag{D}{v},y]$, and let $T_{\maxnum}^v$
  be the rooted decomposition tree of $D_{\maxnum}^v$.
  \end{itemize}
 For each $1\le j\le \eta$, let  $\alpha_j^v:=\max \{f_1(D_j^v,w) \mid \textrm{$w$ is a non-root node of $T_j^v$}\}$, and
 we define $D_{j-1}^v$ and $T^v_{j-1}$ as follows:
 \begin{enumerate}
 \item If $\alpha^{v}_j\neq j$, then let $D_{j-1}^v:=D_j^v$ and $T^{v}_{j-1}:=T^{v}_{j}$.  
 \item If $\alpha^{v}_j= j$ and one of the following is satisfied, then let $D_{j-1}^v:=D_j^v$ and $T^{v}_{j-1}:=T^{v}_{j}$.
 \begin{itemize}
 \item $T^v_j$ has a node with at least $3$ children $w$ such that $f_1(D_j^v,w)=j$.
 \item $T^v_j$ has two incomparable nodes $v_1$ and $v_2$ where $v_1$ is a $j$-critical node $v_1$ and $f_1(D_j^v,v_2)=j$.
 \item $T^v_j$ has no $j$-critical nodes.
 \end{itemize}
 \item Otherwise, $T^v_j$ has the unique $j$-critical node $v_c$. 
 In this case, we choose an unmarked vertex $y$ of $D_j^v$
    represented by $\zeta_2(D_j^v, v_c)$ and
  let $D_{j-1}^v:=\limbtil_{D_j^v}[\bag{D_j^v}{v_c},y]$ and let $T^{v}_{j-1}$ be the rooted decomposition tree of $D_{j-1}^v$.
\end{enumerate}
Lastly for each $0\le j\le \eta$, let $\beta_j^v:=\lrw(\origin{D_j^v})$.

The existence of the unique $j$-critical node in (3) is verified in the next proposition.

\begin{proposition}\label{prop:least} Let $0\leq j \leq \maxnum$ and let $v$ be a non-root node of $T$ such that $\alpha_j^v\leq j$ and $T_j^v$ contains 
  neither
\begin{itemize}
\item a node having at least $3$ children $w$ with $f_1(D_j^v,w)=\alpha_j^v$, nor
\item two incomparable nodes $v_1$ and $v_2$ having the property that $v_1$ is an $\alpha_j^v$-critical node and $f_1(D_j^v,v_2)=\alpha_j^v$.
\end{itemize}
Let $w$ be an $\alpha_j^v$-critical node of $T_j^v$.  Then $w$ is the unique $\alpha_j^v$-critical vertex of $T_j^v$.  Moreover, $\lrw(\origin{D_j^v})= \alpha_j^v+1$ if and only if $\lrw(\origin{D_{j-1}^v})=f_2(D_j^v,w)= \alpha_j^v$.
\end{proposition}

\begin{proof}
Let $k:=\alpha_j^v$.
  We first show that $w$ is the unique $k$-critical node of $T_j^v$.  Let $w'$ be a $k$-critical node of $T_j^v$ that is distinct from $w$.  From the second assumption, $w$ and $w'$ must be comparable
  in $T_j^v$.  Without loss of generality, we may assume that $w$ is a descendant of $w'$ in $T_j^v$.  Then by the definition of $k$-criticality, $w'$ has a child $w''$ such that
  $f_1(D_j^v,w'')=k$ and $w$ is not a descendant of $w''$ in $T_j^v$, contradicting to the second assumption.

  Now we claim that $\lrw(\origin{D_j^v})= k+1$ if and only if $f_2(D_j^v,w)= k$.  By the assumption on $k$ and by Theorem \ref{thm:main}, $\lrw(\origin{D_j^v})\le k+1$.  
  Let $w_1$ and $w_2$ be the two children of $w$ such that $f_1(D_j^v,w_1)=f_1(D_j^v,w_2)=k$. By assumption, every other child $w'$
  of $w$ satisfies that $f_1(D_j^v,w')\leq k-1$. 

 If $f_2(D_j^v,w)=k$, then clearly we have $\lrw(\origin{D_j^v})\ge k+1$ by Theorem~\ref{thm:main}.
  For the forward direction, suppose that $\lrw(\origin{D_j^v})\ge k+1$.  Since $T_j^v$ contains no node having at least three children $w$ such that $f_1(D_j^v,w)=k$, by Theorem \ref{thm:main},
  there should exist a $k$-critical node $v_c$ of $T_j^v$ such that $f_2(D_j^v,v_c)=k$.  Since $w$ is the unique $k$-critical node of $T_j^v$, $w=v_c$ and
  $f_2(D_{j}^v,w)= \lrw(\origin{D_{j-1}^v})=k$, as required.
\end{proof}

Let $v$ be a non-root node of $T$.  From Theorem~\ref{thm:main}, we can easily observe that
$\alpha_{\maxnum}^v\le \lrw(\origin{D_{\maxnum}^v)}) \le \alpha_{\maxnum}^v+1$.  By Proposition~\ref{prop:least}, 
if $T_{\maxnum}^v$ has no unique critical node, then it is easy to determine $\beta_{\maxnum}^v$, 
and otherwise
the computation of $\beta_{\maxnum}^v$ can be reduced to
the computation of $f_2(D_{\maxnum}^v,v_c)$ where $v_c$ is the unique $\alpha_{\maxnum}^v$-critical node of $T_{\maxnum}^v$.  
In order to compute it, we can recursively call the algorithm on $\origin{D_{\alpha_{\maxnum}^v-1}^{v}}$.
However, we will prove that these recursive calls are not needed if we store the values $\beta_j^v$.

\begin{lemma}\label{lem:pdpres}
  Let $v$ be a non-root node of $T$.  Let $i$ be an integer such that $0\le i< \maxnum$.  If $\alpha^{v}_i\le i$, then $\alpha^{v}_{i+1}\le i+1$.
\end{lemma}

\begin{proof}
  Suppose that $\alpha^{v}_{i+1}\ge i+2$.  By the definition of $D^{v}_{i}$, $D^{v}_{i}=D^{v}_{i+1}$ and therefore, $\alpha^{v}_i\ge i+2$, which yields a contradiction.
\end{proof}

\vskip 0.2cm
\noindent{\textbf{Our algorithm.} 
 Now we are ready to present and analyze our algorithm. We describe the algorithm explicitly in Algorithm~\ref{algo:computelrw}.  First, we modify the given decomposition as follows.  For the
 canonical decomposition $D'$ of a distance-hereditary graph $G$, we modify $D'$ into a canonical decomposition $D$ of a connected distance-hereditary graph by adding a root bag $R$ and making it adjacent
 to a bag $R'$ of $D'$ so that $f_1(D, v)=\lrw(G)$, where $v$ is the node corresponding to the bag $R'$.  
 We
 call $(D,R)$ a \emph{modified canonical decomposition of $G$}.  
 Let $T$ be the decomposition tree of the new canonical decomposition $D$.
The basic strategy is to compute $\beta_i^v=\lrw(\origin{D_i^v})$ for all non-root nodes $v$ of $T$ and all integers $i$ such that
 $\alpha^{v}_i\le i$.
 We recall that $\eta=\lfloor \log_2\abs{V(G)}\rfloor$.


We first present the subroutine {\bf Limb} which computes a canonical limb and its decomposition tree recursively.


\begin{small}
\begin{algorithm}[!h]
\DontPrintSemicolon 
\KwIn{A canonical decomposition $D$ of a connected distance-hereditary graph, its rooted decomposition tree $T$ with the root $r$, $\{\gamma(v)\in \bN \mid v\in V(T\setminus r)\}$, a non-root node $w$ of $T$, and $z\in \{1,2\}$. }
\KwOut{A canonical decomposition $D'$ of $D$ associated with $\cT_z[D,w]$, its rooted decomposition tree $T'$ with the root $r'$, $\{\gamma(v)\mid v\in V(T'\setminus r')\}$, and $\alpha$.}
Let $w'$ be the parent of $w$;\;
	\lIf{$z=1$}{choose an unmarked vertex $y$ of $D$ represented by $\zeta_1(D, w)$ and $v\gets w'$;} 
    	\lElse{choose an unmarked vertex $y$ of $D$ represented by $\zeta_2(D, w)$ and $v\gets w$;}
$D'\gets \limbtil_D[\bag{D}{v}, y]$ and obtain $T'$ from $T$ and assign the root $r'$ of $T'$;\;
$\alpha\gets \max\{\gamma(v)\mid$ $v\in V(T'\setminus r')\}$;\;
\Return{$(D', T', \{\gamma(v)\mid v\in V(T'\setminus r')\}, \alpha)$};\;
\caption{{\bf Limb($D,T,\{\gamma(v)\mid v\in V(T\setminus r)\},w\in V(T\setminus r), z\in \{1,2\}$)}. }
\label{algo:computelimb}
\end{algorithm}
\end{small}

We describe the main algorithm in Algorithm~\ref{algo:computelrw}.

\begin{small}
\begin{algorithm}[!h]
\DontPrintSemicolon 
\KwIn{A connected distance-hereditary graph $G$.}
\KwOut{The linear rank-width of $G$.}
Compute a modified canonical decomposition $(D,R)$ of $G$, and the decomposition tree $T$ of $D$ with the root node $r$;\;
Let $\beta_i^v\gets 0$ for each non-root node and each $0\leq i \leq \maxnum$;\;
For each non-root leaf node $v$ in $T$ and each $0\le i\le \maxnum$, let $\beta^{v}_i\gets 1$ ;\;\label{line:leafbag}
$\Gamma\gets \{\beta_i^v\mid v\in V(T\setminus r) \}$;\;
\While{$T$ has a non-root node $v$ where $\beta^{v}_{\maxnum}$ is not computed} 
	{
    Let $v$ be a non-root node in $T$ where $\beta^{v}_{\maxnum}=0$, but $\beta_{\maxnum}^{v'}\ne 0$ for each child $v'$ of $v$;\;\label{line:set1}
    $(D^v_{\maxnum}, T_{\maxnum}^v, \Gamma^v_{\maxnum}, \alpha^v_{\maxnum})\gets $ {\bf Limb($D,T,\Gamma,v,1$)};\;
    Let $S$ be a stack;\label{line:set2}
    \quad $i\gets \alpha^{v}_{\maxnum}$; \quad $k\gets \alpha^{v}_{\maxnum}$;\;
    \While{(true)}
    	{\label{line:computelimbs}
    	\If{($T^v_i$ has a node having at least $3$ children $v'$ with $\beta^{v'}_i=i$) or ($T^v_i$ has two incomparable nodes $v_1$ and $v_2$ having the property that $v_1$ is an $i$-critical node and $\beta^{v_2}_i=i$) or ($T^v_i$ has no $i$-critical nodes)}{
    	\label{line:stoploop}
    	Stop this loop
    	}
    	Find the unique $i$-critical node $v_c$ of $T^v_i$;\;
        $(D^v_{i-1}, T_{i-1}^v, \Gamma^v_{i-1}, \alpha^v_{i-1})\gets $ {\bf Limb($D^v_i,T^v_i,\Gamma^v_i,v_c,2$)};\; 
        $\push(S,i)$ and $i\gets \alpha^{v}_{i-1}$;\;    	  \label{line:set3}
        }
    \lIf{
    	\label{line:leastbag}
    ($T^v_i$ has a node having at least $3$ children $v'$ with $\beta^{v'}_i=i$) or ($T^v_i$ has two incomparable nodes $v_1$ and $v_2$ with the property that $v_1$ is an $i$-critical node and $\beta^{v_2}_i=i$)} {  
    	$\beta^{v}_i\gets i+1$;} 
    \lElse
    	{
    	$\beta^{v}_i\gets i$;
    	}
	\While{$(S \ne \emptyset)$}
		{
		\label{line:computelrw1}
		$j\gets \pull(S)$;\;
    	\lIf{$\beta^{v}_i = j$}{$\beta^{v}_j \gets j+1$;} 
    	\lElse{$\beta^{v}_j \gets j$;}
    	\For{$\ell\gets i+1$ \textbf{to} $j-1$} {
   		$\beta^{v}_{\ell}\gets\beta^{v}_i$;
   	 	 	}
        $i\gets j$;\;
		}
	\For{$j\gets k+1$ \textbf{to} $\maxnum$}{
	\label{line:computelrw2}
	$\beta^{v}_j\gets \beta^{v}_k$;\label{line:loopend}
	}
	}
Let $r'$ be the unique neighbor of the root and \Return{$\beta^{r'}_{\maxnum}$};\;
\caption{{\sc Compute Linear Rank-Width of Connected Distance-Hereditary Graphs}}
\label{algo:computelrw}
\end{algorithm}
\end{small}


\vskip 0.2cm
\noindent{\textbf{Correctness of the algorithm.} 
The following proposition has a key role in the algorithm. It mainly uses the results in Section~\ref{sec:lemmaonCL}.
\begin{proposition}\label{prop:main}   
  Let $v$ be a non-root node of $T$ and let $0\le i\le \maxnum$ such that $\alpha^{v}_i\le i$.  If $w$ is a non-root node of $T^v_i$, then, $\beta^w_i=f_1[D^v_i,w]$.
\end{proposition}


		

\begin{proof}
Let $w$ be a non-root node of $T^v_i$.
By Fact ~\ref{fact:non-root}, for each $i+1\le j\le \maxnum$, $w\in V(T^v_j)$ and hence $w\in V(T)$. Moreover, since $\alpha^{v}_i\le i$, by
  Lemma~\ref{lem:pdpres}, $\alpha^{v}_j\le j$ for all $i+1\le j\le \maxnum$.  
  For each $i\le j\le \maxnum$, we define that
  \begin{itemize}
  \item $y_j$ is an unmarked vertex of $D^v_j$ represented by the marked vertex $\zeta_1(D_j^v, w)$.
  \end{itemize}
  Now, we claim that for each $i\le j\le \maxnum$,  
  \begin{itemize}
  \item $\limbtil_{D_j^v}[\pbag{D_j^v}{w},y_j]$ is locally equivalent to $D_j^w$.  
  \end{itemize}
 If it is true, then we obtain that $\limbtil_{D_i^v}[\pbag{D_i^v}{w},y_i]$ is locally equivalent to $D_i^w$, which implies that $\beta^w_i=f_1[D^v_i, w]$. 
 We prove it by
  induction on $\maxnum -j$.  

  If $j=\maxnum$, then both $D_{\maxnum}^v$ and $D_{\maxnum}^w$ are canonical limbs of $D$.
  Since $w$ is a non-root node of $T^v_{\maxnum}$, $V(\bag{D}{w})$ induces a bag in $D_{\maxnum}^v$, 
  and hence by Proposition \ref{prop:order2}, 
  $D_{\maxnum}^w$ is locally equivalent to $\limbtil_{D_{\maxnum}^v}[\pbag{D_{\maxnum}^v}{w},y_{\maxnum}]$.


  Now let us assume that $i\le j< \maxnum$.  By
  induction hypothesis $D_{j+1}^w$ is locally equivalent to $\limbtil_{D_{j+1}^v}[\pbag{D_{j+1}^v}{w},y_{j+1}]$. Assume first that $\alpha_{j+1}^v\leq j$. Then, by Proposition \ref{prop:order2}, we have that $\alpha_{j+1}^w\leq j$. In that case, by the definition, we have $D_j^v=D_{j+1}^v$ and $D_j^w=D_{j+1}^w$, and we conclude the statement.
  
  Assume now that $\alpha^{v}_{j+1}=j+1$.  Since $\alpha^{v}_{j+1}=j+1$ and $\alpha^{v}_j\le j$, $T^v_{j+1}$ should have a unique $(j+1)$-critical node $v_c$ such that
  $D_j^v=\limbtil_{D_{j+1}^v}[\bag{D_{j+1}^v}{v_c},y_c]$ for some unmarked vertex $y_c$ of $D_{j+1}^v$ represented by $\zeta_2(D_{j+1}^v, v_c)$. 
   We distinguish two cases: either $v_c$ is incomparable with $w$ in $T^v_{j+1}$,
  or $v_c$ is a descendant of $w$ in $T^v_{j+1}$. Since $w$ is a node of $T_j^v$, $w$ cannot be a descendant of $v_c$.

  \vskip 0.2cm
  \noindent\emph{\textbf{Case 1.}  $v_c$ is incomparable with $w$ in $T^v_{j+1}$.}
  
  
  Since $v_c$ is incomparable with $w$ in $T^v_{j+1}$ and $v_c$ is the unique $(j+1)$-critical node in $T^v_{j+1}$, there is no $(j+1)$-critical node in $T_{j+1}^w$. Hence, $D_j^w=D_{j+1}^w$ by definition.  Also, by Proposition \ref{prop:order2},
  \begin{itemize}
  \item $\limbtil_{D_j^v}[\pbag{D_{j}^v}{w},y_j]$ is locally equivalent to $\limbtil_{D_{j+1}^v}[\pbag{D_{j+1}^v}{w},y_{j+1}]$. 
  \end{itemize}
  Hence, we can conclude that $D_j^w$ is locally equivalent to
  $\limbtil_{D_j^v}[\pbag{D_{j}^v}{w},y_j]$ because $D_{j+1}^w$ is locally equivalent to $\limbtil_{D_{j+1}^v}[\pbag{D_{j+1}^v}{w},y_{j+1}]$.

  \vskip 0.2cm
  \noindent\emph{\textbf{Case 2.}   $v_c$ is a descendant of $w$ in $T^v_{j+1}$.}
  
  If $v_c$ is a child of $w$ in $T^v_{j+1}$ and the bag $\bag{D_{j+1}^v}{w}$ has size $3$, then $T_j^v$ cannot contain $w$ as a node, and this contradicts the assumption that $w$ is a node of
  $T_j ^v$.  Therefore, we may assume that either
   \begin{enumerate}
   \item $\abs{\bag{D_{j+1}}{w}}\ge 4$, or
   \item $\abs{\bag{D_{j+1}}{w}}=3$ and $v_c$ is not a child of $w$ in $T_{j+1}^v$.
   \end{enumerate}
   This implies that $v_c$ is a node of the decomposition tree of $\limbtil_{D_{j+1}^v}[ \pbag{D_{j+1}^v}{w},y_{j+1}]$. Let $D':=\limbtil_{D_{j+1}^v}[ \pbag{D_{j+1}^v}{w},y_{j+1}]$.
    By induction hypothesis,
   we know that $D_{j+1}^w$ is locally equivalent to $D'$.
   Note that, by definition $v_c$ is also the unique critical node of $T_{j+1}^w$, and 
   \begin{itemize}
   \item $D_j^w=\limbtil_{D_{j+1}^w}[\bag{D_{j+1}^w}{v_c}, z]$
   for some unmarked vertex $z$ of $D_{j+1}^w$ represented by $\zeta_2(D_{j+1}^w, v_c)$.
   \end{itemize} 
  Also, by Proposition~\ref{prop:order}, 
  \begin{itemize}
  \item $\limbtil_{D_j^v}[\pbag{D_{j}^v}{w},y_j]$ is locally
   equivalent to $\limbtil_{D'}[\bag{D'}{v_c},z']$ where $z'$ is an unmarked vertex of $D'$ represented by $\zeta_2(D', v_c)$.  
   \end{itemize}
   Since $D'$ is locally equivalent to $D^w_{j+1}$, 
   $\limbtil_{D_j^v}[\pbag{D_{j}^v}{w},y_j]$ is locally equivalent to $D_j^w$, and this concludes the proof.
 \end{proof}


\begin{proof}[Proof of Theorem \ref{thm:mainalg}]  We first show that
  Algorithm~\ref{algo:computelrw} correctly computes the linear rank-width of $G$.  If $\abs{V(G)}\le 1$, then $\lrw(G)=0$ from the definition. We may assume that $\abs{V(G)}\ge 2$. Let $(D,R)$ be a modified canonical decomposition of $G$ and let $T$ be the canonical decomposition tree of $D$ and let $r'$ be the unique neighbor of the root of
  $T$.  As we observed, we have that $\lrw(G)=\lrw(\origin{D^{r'}_{\maxnum}})=\beta_{\maxnum}^{r'}$, and want to prove that Algorithm~\ref{algo:computelrw} correctly outputs $\beta_{\maxnum}^{r'}$.
  We claim that for each non-root node $v$ of $T$ and $0\leq i \leq \maxnum$ such that $\alpha^v_i\le i$, Algorithm~\ref{algo:computelrw} correctly computes $\beta_i^v$.  

  Suppose $v$ is a non-root leaf node of $T$.  Since every canonical limb is connected by Lemma~\ref{lem:connected} and $\abs{V(G)}\ge 2$, $D_{\maxnum}^v$ is isomorphic to either a complete graph or a star with at
  least two vertices.  Thus, $\lrw(\origin{D_{\maxnum}^v})=1$, and by construction for each $0\leq i \leq \maxnum$, $D_i^v=D_{\maxnum}^v$, and Line~\ref{line:leafbag} correctly puts these values.
	 	 
  We assume that $v$ is a non-root node in $T$ that is not a leaf, and for all its descendants $v'$ and integers $0\leq \ell \leq \maxnum$ with $\alpha^{v'}_{\ell}\le \ell$, $\beta^{v'}_{\ell}$ is
  computed (\ie $\beta_{\ell}^{v'}\ne 0$).  We claim that Line~\ref{line:computelimbs}-\ref{line:set3} recursively computes $D_i^{v}$ for each $i$ where $\alpha^v_i\le i$.  We first remark that for
  computing $\alpha^v_i$ of $T^v_i$, we use the fact that for each non-root node $w$ of $T^v_i$,  $\beta^w_i=f_1[D^v_i,w]$ from Proposition~\ref{prop:main}.  So,
  $\alpha_i^v=\max\{\beta_i^w\mid$ $w$ a non-root node $w$ of $T^v_i\}$.

  Let $i\in \{0,1, \ldots, \maxnum\}$ such that $\alpha^{v}_{i}\le i$. 
   If $\alpha^{v}_{i}< i$, then by the definition, $T^v_{i-1}=T^v_{i}$ and thus, we take $D_{i-1}^v=D_i^v$. We may assume that
  $\alpha^v_i=i$.  If either $T^v_i$ has a node with at least $3$ children $v'$ such that $\beta^{v'}_i=i$, or $T^v_i$ has two incomparable nodes $v_1$ and $v_2$ with $v_1$ an $i$-critical node and
  $\beta^{v_2}_i=i$, then from the definition of $D^v_i$, we have that $D^v_{i-1}=D^v_{i}$ and for all $0\le \ell\le i-1$, $\alpha^v_{\ell}=i>\ell$.  Since we do not need to evaluate $\beta^v_{\ell}$
  when $\alpha^v_{\ell}>\ell$, we stop the loop.  If $T^v_i$ has no $i$-critical node, then $\beta^v_{i}=\alpha^v_{i}=i$, that is, the $\beta_i^v$ value cannot be increased by one. In this case, we
  also stop the loop.  These $3$ cases are the conditions in Line~\ref{line:stoploop}.


  Suppose neither of the conditions in Line~\ref{line:stoploop} occur.  Then by Proposition~\ref{prop:least}, $T^v_i$ has a unique $i$-critical bag $v_c$ and $D^v_{i-1}$ is equal to a canonical
  limb $\limbtil_{D_i^v}[\bag{D_{i}^v}{v_c},y]$ where $y$ is some unmarked vertex of $D_i^v$ represented by $\zeta_2(D_i^v, v_c)$.  So, we compute
  $D_{i-1}^v$ from $D_i^v$, the rooted decomposition tree $T_{i-1}^v$ of $D_{i-1}^v$ and compute subsequently $\alpha^v_{i-1}$.  Notice that for all $\alpha^v_{i-1}\le \ell \le i-1$,
  $D^v_{\ell}=D^v_{i-1}$ and thus it is sufficient in the next iteration to deal with $D^v_{\alpha^v_{i-1}}$ directly.  Thus, Line~\ref{line:computelimbs}-\ref{line:set3} correctly computes
  canonical decompositions $D_i^v$ for each $i$ where $\alpha^v_i=i$.
  
  Now we verify the procedure of computing $\beta^v_j$ in Line~\ref{line:leastbag}.  Let $0\leq \ell \leq \maxnum$ be the minimum integer such that $\alpha^v_{\ell}=\ell$.  If $\ell=0$, then
  $\beta^v_{\ell}=1$.  Suppose $\ell\ge 1$. Then since $\alpha^v_{\ell-1}>\ell-1$, by Theorem~\ref{thm:main}, we have that
  \begin{enumerate}
  \item $\beta^v_{\ell}=\ell+1$ if either $T^v_{\ell}$ has a node having at least $3$ children $v'$ with $\beta^{v'}_{\ell}=\ell$, or two incomparable nodes $v_1$ and $v_2$ with the property that 
    $v_1$ is an $i$-critical node and $\beta^{v_2}_i=i$,
  \item $\beta^v_{\ell}=\ell$ if otherwise.
  \end{enumerate}
  Thus, Line~\ref{line:leastbag} correctly computes it.

  In the loop in Line~\ref{line:computelimbs}, we use a stack to pile up the integers $i$ such that $T^v_i$ has the unique $i$-critical node.  When $T^v_i$ has the unique $i$-critical node, by
  Proposition ~\ref{prop:least},
  \begin{enumerate}
  \item $\beta^v_i=i+1$ if $\beta^v_{i-1}=i$, and
  \item $\beta^v_i=i$ if $\beta^v_{i-1}\le i-1$.
  \end{enumerate}
  So, from the lower value in the stack we can compute $\beta^v_i$ recursively.  From Line~\ref{line:computelrw1} to Line~\ref{line:computelrw2}, Algorithm~\ref{algo:computelrw} computes all
  $\beta^v_i$ correctly where $\alpha^v_i\le i$, and in particular, it computes $\beta^v_{\maxnum}$.  Therefore, at the end of the algorithm, it computes $\beta^{r'}_{\maxnum}$ that is equal to the linear rank-width
  of $G$.
	
	\medskip
	  \medskip
  \noindent{\textbf{The running time of the algorithm.}}
  Let us now analyze its running time. Let $n$ and $m$ be the number of vertices and edges of $G$. Its canonical decomposition can be computed in time $\cO(n+m)$ by Theorem \ref{thm:CED}, and one can compute
  a modified canonical decomposition $(D,R)$ in constant time.
  Note that  the number of bags in $D$ is bounded by $\cO(n)$ (see \cite[Lemma 2.2]{GavoilleP03}).
  
  We first remark that Algorithm~\ref{algo:computelimb} runs in time $\mathcal{O}(n)$. This is because when we take a limb from a canonical decomposition, 
  we need to take a local complementation or a pivoting on a sub-decomposition, and in the worst case, we may visit each bag to apply these operations.
  The decomposition tree and $\alpha$, $\beta$ values can be obtained in linear time.
  
  Now we observe the running time of Algorithm~\ref{algo:computelrw}. 
  The number of iterations of the whole loop from Line~\ref{line:set1} to Line~\ref{line:loopend}  is at most $\mathcal{O}(n)$
  because it runs in as many as the number of bags in $D$.
   Lines~\ref{line:set1}-\ref{line:set2} can be implemented in time $\cO(n)$. 
  The loop in Line~\ref{line:computelimbs} runs $\log_2(n)$ times because $\lrw(G)\leq \log_2(n)$, and all the steps in Line~\ref{line:computelimbs} can be implemented in time $\cO(n)$.
  Also, Lines~\ref{line:leastbag}-\ref{line:computelrw2} can be done in time $\cO(n)$.  We conclude that this
  algorithm runs in time $\cO(n^2\cdot \log_2 n)$.
  
  \medskip
  \medskip
  \noindent{\textbf{Finding an optimal linear layout.}}
  We finally establish how to find a linear layout witnessing $\lrw(G)$. We may assume that $G$ has at least $3$ vertices.
We can assume that for each non-root node $v$ of $T$ and $0\leq i \leq \maxnum$ with $\alpha^v_i\le i$, $T^v_i$ and $\beta^v_i$ are computed.
We inductively obtain optimal linear layouts of $\origin{D^v_i}$ using those values.
If $v$ is a non-root leaf node of $T^v_i$, then $\origin{D^v_i}$ is either a complete graph or a star for all $i$, and thus, any ordering of $V(\origin{D^v_i})$ is a linear layout of width $1$. 
We may assume that $v$ is a not a leaf node.

We will search for the path depicted in Lemma~\ref{lem:sdpath2} to apply the same technique used in the proof of Theorem~\ref{thm:main}.
What we have shown in Theorem~\ref{thm:main} is that for a canonical decomposition $D$ of a distance-hereditary graph with its decomposition tree $T_D$, 
if $T_D$ has a path $P:=v_0v_1 \cdots v_nv_{n+1}$ such that 
 \begin{itemize}
 \item for each node $v$ in $P$ and a component $T$ of $D\setminus V(\bag{D}{v})$ not containing a bag $\bag{D}{w}$ with
  $w\in P$, $f(\bag{D}{v},T)\le k-1$,
  \end{itemize}
then we can generate a linear layout of $\origin{D}$ having width at most $k$. 
But it assumed that we have a linear layout of graphs corresponding to pending subtrees.
So, for our purpose, it is necessary to find such a path with $k=\beta^v_i$ such that
 \begin{itemize}
 \item for each node $v$ in $P$ and a component $T$ of $D\setminus V(\bag{D}{v})$ not containing a bag $\bag{D}{w}$ with
  $w\in P$, a linear layout of $\limbhat_D[\bag{D}{v},T]$ with an optimal width is already computed.
  \end{itemize}
 
Let us assume that $k=\beta^v_i$. There are two cases; either $T^v_i$ has the $k$-critical node or not.

	  \vskip 0.2cm
\noindent\emph{\textbf{Case 1.} $T^v_i$ has no $k$-critical node.}

In this case, we take a path $P$ from the root node of $T^v_i$ (or both end nodes of the root edge) to a node $w$ where $\beta^w_i=k$ but for every descendant $w'$ of $w$, $\beta^{w'}_i<k$.
Since $T^v_i$ has no $k$-critical node, every node outside of this path has $\beta$ value less than $k$.
Thus, the graphs corresponding to subtrees pending to this path have linear rank-width at most $k-1$, and moreover, by induction hypothesis, 
we already obtained an optimal linear layout for each graph. This path can be computed in linear time.
  
  	  \vskip 0.2cm
\noindent\emph{\textbf{Case 2.} $T^v_i$ has a $k$-critical node.}

	First note that $T^v_i$ cannot have two $k$-critical nodes, otherwise, $\beta^v_i=k+1$, which contradicts to our assumption.
   Let $x$ be the unique $k$-critical node of $T^v_i$, and let $x_1, x_2$ be two children of $x$ where $\beta^{x_j}_i=k$ for each $j\in \{1,2\}$.
   For each $j\in \{1,2\}$, we choose a descendant $w_j$ of $x_j$ where $\beta^{w_j}_i=k$ but for every descendant $w_j'$ of $w_j$, $\beta^{w_j}_i<k$. 
   Let $P$ be the path from $w_1$ to $w_2$ in $T^v_i$. 
   This path can be computed in linear time.
   
   Since $x$ is the unique $k$-critical node of $T^v_i$, 
   every node below of this path has $\beta$ value less than $k$, and the graphs corresponding to subtrees pending to the path are computed in advance.
   Moreover, since this case is exactly when $\alpha^v_i=k$ and $\beta^v_i=k$ and $T^v_i$ has a unique critical node, 
   the canonical decomposition corresponding to the subtree of $T^v_i\setminus x$ containing the parent of $x$ is exactly $D^v_{k-1}$, 
   and $\origin{D^v_{k-1}}$ should have linear rank-width at most $k-1$ as $\beta^v_i=k$.
   By the induction hypothesis, the optimal linear layout of $\origin{D^v_{k-1}}$ is also computed before, as required.
   
   \medskip
   We conclude that we can compute an optimal layout of $G$ in time $\mathcal{O}(n^2\cdot \log_2 n)$.
   \end{proof}




\section{Path-width of matroids with branch-width $2$}\label{sec:matroid}

As a corollary of Theorem~\ref{thm:mainalg}, we can compute the path-width of matroids of branch-width at most $2$. We first recall the necessary materials about matroids.  We refer to the book written by Oxley~\cite{Oxley12} for our matroid notations.

	  \vskip 0.2cm
\noindent{\textbf{Matroids.}}
A pair $(E(M),\cI(M))$ is called a \emph{matroid} $M$ if $E(M)$, called the \emph{ground set of $M$}, is a finite set and $\cI(M)$, called the set of \emph{independent sets of $M$}, is a nonempty collection of
subsets of $E(M)$ satisfying the following conditions:
\begin{enumerate}
\item[(I1)] if $I\in \cI(M)$ and $J\subseteq I$, then $J\in \cI(M)$,
\item[(I2)] if $I,J\in \cI(M)$ and $|I|<|J|$, then $I\cup \{z\}\in \cI(M)$ for
  some $z\in J\backslash I$.
\end{enumerate}
A maximal independent set in $M$ is called a \emph{base of $M$}. It is known that, if $B_1$ and $B_2$ are bases of $M$, then $\abs{B_1}=\abs{B_2}$. 
  
  For a matroid $M$ and a subset $X$ of $E(M)$, we let $(X,\{I\subseteq X: I \in \cI(M)\})$ be the matroid denoted by $\restrict{M}{X}$.  The size of a base of $\restrict{M}{X}$ is called the
\emph{rank} of $X$ in $M$ and the \emph{rank function} of $M$ is the function $r_{M}:2^{E(M)}\to \bN$ that maps every $X\subseteq E(M)$ to its rank. The rank of $E(M)$ is called the rank of $M$. 

If $M$ is a matroid, then we define $\lambda_{M}$, called the \emph{connectivity function of $M$}, such that for
every subset $X$ of $E(M)$, 
\[\lambda_{M}(X) = r_{M}(X) + r_{M}(E(M)\setminus X) -r_{M}(E(M)) +1.\] It is known that the function $\lambda_{M}$ is symmetric and submodular.

Let $A$ be a binary matrix and let $E$ be the column labels of $A$. Let $\cI$ be the collection of all those subsets $I$ of $E$ such that the columns of $A$ with index in $I$ are linearly
independent. Then $M(A):=(E,\cI)$ is a matroid. Every matroid isomorphic to $M(A)$ for some matrix $A$ is called a \emph{binary matroid} and $A$ is called a \emph{representation of $M$ over the
  binary field}.

We now define fundamental graphs of binary matroids.  Let $G$ be a bipartite graph with a bipartition $(A, B)$.  We define $M(G, A, B)$ as the binary matroid represented by the $(A\times V)$-matrix
$(I_A \quad A_G[A, B])$ where $I_A$ is the $(A\times A)$ identity matrix; and we call $G$ a \emph{fundamental graph} of $M(G,A,B)$.  We remark that $\abs{E(M)}=\abs{V(G)}$.


	  \vskip 0.2cm
\noindent{\textbf{Branch-width and path-width of matroids.}}
 A \emph{branch-decomposition} of a matroid $M$ is a pair $(T,L)$, where $T$ is a subcubic tree and $L$ is a bijection from the elements of $E(M)$ to the leaves of $T$. For an edge $e$ in $T$, $T\setminus e$ induces a partition $(X_{e} ,Y_{e} )$ of the leaves of $T$. The \emph{width} of an edge $e$ is defined as $\lambda_{M} (L^{-1}(X_{e} ))$. The \emph{width} of a branch-decomposition $(T,L)$ is the maximum width over all edges of $T$. The \emph{branch-width} of $M$, denoted by $\bwd{M}$, is the minimum width over all branch-decompositions of $M$. If $\abs{E(M)}\leq 1$, then $M$ admits no branch-decomposition and $\bwd{M}=0$.

A sequence $e_1,\ldots,e_n$ of the ground set $E(M)$ is called a \emph{linear layout} of $M$.  The \emph{width} of a
linear layout $e_1,\ldots,e_n$ of $M$ is
\[\max_{1\le i\le n-1}\{\lambda_{M}(\{e_1,\ldots,e_i\})\}.\]
The \emph{path-width} of $M$, denoted by $\lbwd{M}$, is defined as the minimum width over all linear layouts of $M$.

The following relation is established by Oum~\cite{Oum05}.

\begin{proposition}[Oum~\cite{Oum05}]\label{prop:matroidrank}
Let $G$ be a bipartite graph with a bipartition
$(A, B)$ and let $M: = M(G, A,B)$. 
For every $X\subseteq V(G)$, $\cutrk_G(X)=\lambda_{M}(X)-1$.
Thus, $\rw(G)=\bwd{M}-1$ and $\lrw(G)=\lbwd{M} -1$.
\end{proposition}


Here, we observe that 
every matroid of branch-width at most $2$ is binary. 
This can be observed from the known minor characterizations for binary matroids and matroids of branch-width at most $2$.
For the definition of matroid minors, we refer to \cite{Oxley12}.

\begin{theorem}[Tutte~\cite{Tutte1958, Tutte1965}]\label{thm:excludedbinary}
A matroid is binary if and only if it has no minor isomorphic to $U_{2,4}$.
\end{theorem}

\begin{theorem}[Dharmatilake~\cite{Dharmatilake1996}]\label{thm:excludedbw2}
A matroid has branch-width at most $2$ if and only if it has no minor isomorphic to $U_{2,4}$ and $M(K_4)$.
\end{theorem}


\begin{corollary}\label{cor:matroidpathwidth} 
	The path-width of every $n$-element matroid of branch-width at most $2$ can be computed in time $\mathcal{O}(n^2 \cdot \log_2 n)$,
provided that the matroid is given by an independent set oracle. Moreover, 
a linear layout of the matroid witnessing the path-width can be computed with the same time complexity.
\end{corollary}

	\begin{proof}
	Let $M$ be a matroid of branch-width at most $2$ and assume that an independent oracle of $M$ is given.
	By Theorems~\ref{thm:excludedbinary} and \ref{thm:excludedbw2}, every matroid of branch-width at most $2$ is binary.
	We first run a greedy algorithm to find a base $B$ of $M$~\cite[Section 1.8]{Oxley12} in time $\mathcal{O}(\abs{E(M)})$.
   After choosing one base $B$, for each $e\in B$ and $e'\in E(M)\setminus B$, we test whether $(B\setminus \{e\})\cup \{e'\}$ is again a base using the independent set oracle, and we create the fundamental graph $G$ with respect to $M$. It can be done in time $\mathcal{O}(\abs{E(M)}^2)$. 
	Since $M$ has branch-width at most $2$, by Proposition~\ref{prop:matroidrank}, the rank-width of $G$ is at most $1$. Using Theorem~\ref{thm:mainalg}, we can compute the linear rank-width of $G$ in time $\mathcal{O}(\abs{E(M)}^2\cdot \log_2 \abs{E(M)})$, which is the same as $\lbwd{M}-1$.
	Moreover, we can compute a linear layout witnessing $\lrw(G)$ in the same time, that corresponds to the linear layout of $M$ witnessing $\lbwd{M}$.
\end{proof}

\section{An upper bound on linear rank-width}\label{sec:upperbound}

As we promised, we prove the following lemma here. 
We remark that Bodlaender, Gilbert, Hafsteinsson and Kloks~\cite{BodlaenderGK91} proved a similar relation between tree-width and path-width.

\begin{lemma}\label{lem:lrwtrivialbound2}
Let $k$ be a positive integer and let $G$ be a graph of rank-width $k$. 
Then $\lrw(G)\le k \lfloor\log_{2}\abs{V(G)}\rfloor$.
\end{lemma}

\begin{proof}
Since $k$ is a positive integer, we have $\abs{V(G)}\ge 2$.
Let $(T,L)$ be a rank-decomposition of $G$ having width $k$.
For convenience, we choose an edge $e$ of $T$ and subdivide it with introducing a new vertex $x$, and regard $x$ as the root of $T$. 
For each internal vertex $t$ of $T$ with two subtrees $T_1$ and $T_2$ of $T\setminus t$ not containing $x$, 
let $\ell(t):=T_1$ and $r(t):=T_2$ if
the number of leaves of $T$ in $T_1$ is  at least 
the number of leaves of $T$ in $T_2$.
Let $S$ be a linear layout of $G$ satisfying that
\begin{itemize}
\item for each $v_1, v_2\in V(G)$ with the first common ancestor $w$ of $v_1$ and $v_2$ in $T$, $L(v_1)<_S L(v_2)$ if $L(v_1)\in V(\ell(w))$.
\end{itemize}
We can construct such a linear layout inductively. 

We show that $S$ has width at most $k \lfloor\log_{2}\abs{V(G)}\rfloor$.
Let $w$ be a vertex of $G$ that is not the first vertex of $S$ and let $S_w:=\{v:v<_{S} w\}$.
Let $P_w$ be the path from $L(w)$ to the root $x$ in $T$.
Note that for each $t\in V(P_w)\setminus \{L(w)\}$ and the subtree $T'$ of $T\setminus t$ not containing $x$ and $L(w)$, 
\begin{itemize}
\item if $r(t)=T'$, then all leaves of $T$ in $T'$ are not contained in $S_w$,  and 
\item if $\ell(t)=T'$, then all leaves of $T$ in $T'$ are contained in $S_w$.
\end{itemize}
Let $Q$ be the set of all vertices $t$ in $P_w$ except $w$ such that 
the subtree $\ell(t)$ does not contain $x$ and $L(w)$.

Let $q_1, q_2, \ldots, q_m$ be the sequence of all vertices in $Q$ such that for each $1\le j\le m-1$, $q_j$ is a descendant of $q_{j+1}$ in $T$, and
let $Q_i$ be the set of all leaves of $T$ contained in $\ell(q_i)$.
Clearly, $S_w=Q_1\cup Q_2\cup \cdots \cup Q_m$ and $V(G)\setminus S_w\neq \emptyset$.
Therefore, we have
  \begin{align*}
  \abs{V(G)}&= \abs{Q_1} + \cdots + \abs{Q_m} +\abs{V(G)\setminus S_w}  \\
  &\ge 1+2+4+\cdots 2^{m-1}+1 \\
  &=2^m. 
    \end{align*} 
Thus, $m\le \lfloor\log_2 \abs{V(G)}\rfloor$.

  Note that for each $1\le j\le m$, $\rank(A_G[Q_i,V(G)\setminus S_w])\le k$.
  Therefore, we have that 
  \[\cutrk_G(S_w)=\rank(A_G[(Q_1\cup \cdots \cup Q_m, V(G)\setminus S_w)])\le km\le  k\lfloor\log_2 \abs{V(G)}\rfloor.\]
Since $w$ was arbitrarily chosen, it implies that $\lrw(G)\le k\lfloor\log_2 \abs{V(G)}\rfloor$.
\end{proof}

\section{Concluding remarks}

We have provided a characterization of the linear rank-width of distance-hereditary graphs in terms of their canonical decompositions, and use this characterization to derive a polynomial-time algorithm to compute their linear
rank-width. An easy consequence of this is also a polynomial-time algorithm for computing the path-width of matroids of branch-width at most $2$, which was not adressed in the past.  

 Since distance-hereditary graphs are exactly graphs of rank-width at most $1$, natural extensions
are to consider graphs of bounded rank-width.
\begin{itemize}
\item Can we compute in polynomial time the linear rank-width of graphs of rank-width $\leq k$, for fixed $k\ge 2$?
\end{itemize}


In the second part of this work~\cite{AdlerKK152}, we will discuss structural properties of distance-hereditary graphs related to linear rank-width. 
Note that Jeong, Kwon, and Oum~\cite{JeongKO13} provided a lower bound on the size of the vertex-minor obstruction set for graphs with bounded linear rank-width, 
by providing a set of pairwise locally non-equivalent vertex-minor obstructions for graphs of linear rank-width at most $k$ for each $k$.
Their graphs are indeed distance-hereditary graphs, and
we will give a more general way to generate all distance-hereditary vertex-minor obstructions using the characterization given in this paper. 
Also, we prove that for a fixed tree $T$, every distance-hereditary graph of sufficiently large linear rank-width contains $T$ as a vertex-minor.
\section*{Acknowledgment}
The authors would like to thank Sang-il Oum for pointing out that the computation of matroid path-width can be obtained as a corollary of our main result.

\begin{thebibliography}{10}

\bibitem{AdlerK13}
Isolde Adler and Mamadou~Moustapha Kant{\'e}.
\newblock Linear rank-width and linear clique-width of trees.
\newblock {\em Theoret. Comput. Sci.}, 589:87--98, 2015.

\bibitem{AdlerKK152}
Isolde Adler, Mamadou~Moustapha Kant{\'e}, and O-joung Kwon.
\newblock Linear rank-width of distance-hereditary graphs {II}. {V}ertex-minor
  obstructions.
\newblock {\em Submitted}, 2015.
\newblock http://arxiv.org/abs/1508.04718.

\bibitem{BandeltM86}
Hans-J{\"u}rgen Bandelt and Henry~Martyn Mulder.
\newblock Distance-hereditary graphs.
\newblock {\em J. Comb. Theory, Ser. B}, 41(2):182--208, 1986.

\bibitem{BodlaenderGK91}
Hans~L. Bodlaender, John~R. Gilbert, Hjlmtr Hafsteinsson, and Ton Kloks.
\newblock Approximating treewidth, pathwidth, and minimum elimination tree
  height.
\newblock {\em J. Algorithms}, 18:238--255, 1991.

\bibitem{BodlaenderK96}
Hans~L. Bodlaender and Ton Kloks.
\newblock Efficient and constructive algorithms for the pathwidth and treewidth
  of graphs.
\newblock {\em J. Algorithms}, 21(2):358--402, 1996.

\bibitem{Bouchet88}
Andr\'e Bouchet.
\newblock Transforming trees by successive local complementations.
\newblock {\em J. Graph Theory}, 12(2):195--207, 1988.

\bibitem{CourcelleO00}
Bruno Courcelle and Stephan Olariu.
\newblock Upper bounds to the clique width of graphs.
\newblock {\em Discrete Applied Mathematics}, 101(1-3):77--114, 2000.

\bibitem{CunninghamE80}
William~H. Cunnigham and Jack Edmonds.
\newblock A combinatorial decomposition theory.
\newblock {\em Canadian Journal of Mathematics}, 32:734--765, 1980.

\bibitem{Dahlhaus00}
Elias Dahlhaus.
\newblock Parallel algorithms for hierarchical clustering, and applications to
  split decomposition and parity graph recognition.
\newblock {\em Journal of Graph Algorithms}, 36(2):205--240, 2000.

\bibitem{Dharmatilake1996}
Jack~S. Dharmatilake.
\newblock A min-max theorem using matroid separations.
\newblock In {\em Matroid theory ({S}eattle, {WA}, 1995)}, volume 197 of {\em
  Contemp. Math.}, pages 333--342. Amer. Math. Soc., Providence, RI, 1996.

\bibitem{Diestel05}
Reinhard Diestel.
\newblock {\em Graph theory}, volume 173 of {\em Graduate Texts in
  Mathematics}.
\newblock Springer, Heidelberg, fourth edition, 2010.

\bibitem{EllisST94}
Jonathan~A. Ellis, Ivan~Hal Sudborough, and Jonathan~S. Turner.
\newblock The vertex separation and search number of a graph.
\newblock {\em Inf. Comput.}, 113(1):50--79, 1994.

\bibitem{Ganian10}
Robert Ganian.
\newblock Thread graphs, linear rank-width and their algorithmic applications.
\newblock In Costas~S. Iliopoulos and William~F. Smyth, editors, {\em IWOCA},
  volume 6460 of {\em Lecture Notes in Computer Science}, pages 38--42.
  Springer, 2010.

\bibitem{GavoilleP03}
Cyril Gavoille and Christophe Paul.
\newblock Distance labeling scheme and split decomposition.
\newblock {\em Discrete Mathematics}, 273(1-3):115--130, 2003.

\bibitem{GeelenGW2002}
James~F. Geelen, A.~M.~H. Gerards, and Geoff Whittle.
\newblock Branch-width and well-quasi-ordering in matroids and graphs.
\newblock {\em J. Combin. Theory Ser. B}, 84(2):270--290, 2002.

\bibitem{GeelenO09}
James~F. Geelen and Sang\mbox{-}il Oum.
\newblock Circle graph obstructions under pivoting.
\newblock {\em Journal of Graph Theory}, 61(1):1--11, 2009.

\bibitem{GeelenGW2007}
Jim Geelen, Bert Gerards, and Geoff Whittle.
\newblock Excluding a planar graph from {${\rm GF}(q)$}-representable matroids.
\newblock {\em J. Combin. Theory Ser. B}, 97(6):971--998, 2007.

\bibitem{GeelenGW2014}
Jim Geelen, Bert Gerards, and Geoff Whittle.
\newblock Solving {R}ota's conjecture.
\newblock {\em Notices Amer. Math. Soc.}, 61(7):736--743, 2014.

\bibitem{JeongKO13}
Jisu Jeong, O{-}joung Kwon, and Sang{-}il Oum.
\newblock Excluded vertex-minors for graphs of linear rank-width at most {$k$}.
\newblock {\em European J. Combin.}, 41:242--257, 2014.

\bibitem{Kante2012}
Mamadou~Moustapha Kant{\'e}.
\newblock Well-quasi-ordering of matrices under {S}chur complement and
  applications to directed graphs.
\newblock {\em European J. Combin.}, 33(8):1820--1841, 2012.

\bibitem{KanteR2013}
Mamadou~Moustapha Kant{\'e} and Michael Rao.
\newblock The rank-width of edge-coloured graphs.
\newblock {\em Theory Comput. Syst.}, 52(4):599--644, 2013.

\bibitem{Kashyap2008}
Navin Kashyap.
\newblock Matroid pathwidth and code trellis complexity.
\newblock {\em SIAM J. Discrete Math.}, 22(1):256--272, 2008.

\bibitem{THHD93}
Ton Kloks, Hans~L. Bodlaender, Haiko M{\"u}ller, and Dieter Kratsch.
\newblock Computing treewidth and minimum fill-in: All you need are the minimal
  separators.
\newblock In Thomas Lengauer, editor, {\em ESA}, volume 726 of {\em Lecture
  Notes in Computer Science}, pages 260--271. Springer, 1993.

\bibitem{MegiddoHGJP88}
Nimrod Megiddo, S.~Louis Hakimi, M.~R. Garey, David~S. Johnson, and Christos~H.
  Papadimitriou.
\newblock The complexity of searching a graph.
\newblock {\em J. ACM}, 35(1):18--44, 1988.

\bibitem{Oum05}
Sang\mbox{-}il Oum.
\newblock Rank-width and vertex-minors.
\newblock {\em J. Comb. Theory, Ser. B}, 95(1):79--100, 2005.

\bibitem{OumS06}
Sang\mbox{-}il Oum and Paul~D. Seymour.
\newblock Approximating clique-width and branch-width.
\newblock {\em J. Comb. Theory, Ser. B}, 96(4):514--528, 2006.

\bibitem{Oxley12}
James Oxley.
\newblock {\em Matroid Theory}.
\newblock Number~21 in Oxford Graduate Texts in Mathematics. Oxford University
  Press, second edition, 2012.

\bibitem{Tutte1958}
W.~T. Tutte.
\newblock A homotopy theorem for matroids. {I}, {II}.
\newblock {\em Trans. Amer. Math. Soc.}, 88:144--174, 1958.

\bibitem{Tutte1965}
W.~T. Tutte.
\newblock Lectures on matroids.
\newblock {\em J. Res. Nat. Bur. Standards Sect. B}, 69B:1--47, 1965.

\end{thebibliography}

\end{document}